\pdfoutput=1
\documentclass[onefignum,onetabnum]{siamart220329}

\usepackage{amsmath,mathtools}
\usepackage{amssymb}
\usepackage{algorithm,algorithmic}

\usepackage{lipsum}
\usepackage{amsfonts}
\usepackage{graphicx}
\usepackage{epstopdf}
\usepackage{algorithmic}
\usepackage{eqparbox}

\usepackage{tikz}
\usetikzlibrary{arrows}

\usepackage{makeidx}
\usepackage{mathrsfs}
\makeatletter
\renewcommand{\fnum@figure}{Fig. \thefigure}
\makeatother
\usepackage{amssymb}
\usepackage{theorem} 
\usepackage{booktabs}
\usepackage{euscript}
\usepackage{exscale,relsize}
\usepackage{graphicx}
\usepackage{booktabs}
\usepackage{srcltx}
\usepackage{xcolor}
\usepackage{pifont}
\usepackage{url}
\usepackage{amsmath,color}
\usepackage{wasysym}
\usepackage{amssymb}
\usepackage{float}

\definecolor{labelkey}{rgb}{0,0.08,0.45}
\definecolor{refkey}{rgb}{0,0.6,0.0}
\definecolor{Brown}{rgb}{0.45,0.0,0.05}
\definecolor{dgreen}{rgb}{0.00,0.49,0.00}
\definecolor{dblue}{rgb}{0,0.08,0.75}

\ifpdf
  \DeclareGraphicsExtensions{.eps,.pdf,.png,.jpg}
\else
  \DeclareGraphicsExtensions{.eps}
\fi

\renewcommand\theenumi{(\roman{enumi})}
\renewcommand\theenumii{(\alph{enumii})}

\def\HH{\mathcal H}
\def\HHH{\boldsymbol{\mathcal H}}
\def\CCX{\bigtimes_{i\in I}C_i}
\def\CCN{\bigcap_{i\in I}C_i}
\def\CCM{\sum_{i\in I}\omega_iC_i}

\def\DD{\boldsymbol{D}}
\def\bz{\boldsymbol{z}}
\def\bv{\boldsymbol{v}}
\def\by{\boldsymbol{y}}

\def\bx{\boldsymbol{x}}
\def\bp{\boldsymbol{p}}
\def\bc{\boldsymbol{c}}

\def\G0{\Gamma_0(\mathcal G)}
\def\H0{\Gamma_0(\mathcal H)}
\newcommand{\NN}{\ensuremath{\mathbb N}}
\newcommand{\RR}{\ensuremath{\mathbb{R}}}

\newcommand{\menge}[2]{\big\{{#1}~\big |~{#2}\big\}}
\newcommand{\Menge}[2]{\left\{{#1}~\left|~{#2}\right.\right\}}
\newcommand*{\Id}{\text{\normalfont Id}}

\newcommand{\proj}{\ensuremath{\text{\rm Proj}\,}}

\newcommand{\lmo}{\ensuremath{\operatorname{LMO}}}
\newcommand{\dist}{\ensuremath{\text{\rm dist}\,}}

\newcommand{\epi}{\ensuremath{\text{\rm epi}\,}}
\newcommand{\gr}{\ensuremath{\text{\rm gra}\,}}

\newcommand{\Argmin}{\ensuremath{{\text{\rm Argmin}}\,}}
\newcommand{\Scal}[2]{\bigg\langle{#1}\;\bigg|\:{#2}\bigg\rangle}
\newcommand{\scal}[2]{{\langle{{#1}\mid{#2}}\rangle}}

\newcommand{\FW}{\mbox{F-W }}

\newsiamremark{remark}{Remark}
\newsiamremark{fact}{Fact}
\newsiamremark{example}{Example}
\newsiamremark{hypothesis}{Hypothesis}
\crefname{hypothesis}{Hypothesis}{Hypotheses}
\newsiamthm{claim}{Claim}

\headers{Splitting the Conditional Gradient Algorithm}{Zev
Woodstock and Sebastian Pokutta}

\title{Splitting the Conditional Gradient
Algorithm
}

\author{Zev Woodstock\footnotemark[3] \thanks{ Department of
Mathematics and Statistics, James Madison University, Harrisonburg,
VA, USA (\mbox{\email{woodstzc@jmu.edu}}),
\url{https://zevwoodstock.github.io/}).}  
\and Sebastian Pokutta \thanks{
Department of AI in Society, Science,
and Technology, Zuse Institute Berlin, Berlin, DE and
Institute of Mathematics, Technische
Universit\"at Berlin, Berlin, DE
(\mbox{\email{\{woodstock,pokutta\}@zib.de}}).
}}

\usepackage{amsopn}

\ifpdf
\hypersetup{
  pdftitle={Splitting the Conditional Gradient Algorithm},
  pdfauthor={Z. Woodstock and S. Pokutta},
  linktocpage=true,citecolor=dblue,linkcolor=dgreen
}
\fi

\begin{document}

\maketitle

\begin{abstract}
We propose a novel generalization of the conditional gradient (CG /
Frank-Wolfe) algorithm for minimizing a smooth function $f$ under
an intersection of compact convex sets, using a first-order oracle
for $\nabla f$ and linear minimization oracles (LMOs) for the
individual sets. Although this computational framework presents
many advantages, there are only a small number of algorithms which
require one LMO evaluation per set per iteration; furthermore,
these algorithms require $f$ to be convex. Our algorithm appears to
be the first in this class which is proven to also converge in the
nonconvex setting. Our approach combines a penalty method and a
product-space relaxation. We show that one conditional gradient
step is a sufficient subroutine for our penalty method to converge,
and we provide several analytical results on the product-space
relaxation's properties and connections to other problems in
optimization. We prove that our average Frank-Wolfe gap converges
at a rate of $\mathcal{O}(\ln t/\sqrt{t})$, -- only a log factor
worse than the vanilla CG algorithm with one set.
\end{abstract}



\section{Introduction}

Given a smooth function $f$ which maps from a real Hilbert
space $\HH$ to $\mathbb{R}$ and a finite collection of $m$
nonempty compact convex subsets $(C_i)_{i\in I}$ of $\HH$, we seek
to solve the following: 
\begin{equation}
\label{e:p}
\text{minimize}\;\; f(x)\quad\text{subject to}\quad
x\in\bigcap_{i\in I}C_i,
\end{equation}
which has many applications in imaging, signal processing, and data
science \cite{Cens15,Comb11,Fall19}. 
Classical projection-based
algorithms can be used to solve \eqref{e:p} if given access to the
operator $\proj_{\bigcap_{i\in I} C_i}$. However, in practice, 
computing a projection onto $\bigcap_{i\in I} C_i$ is either
impossible or numerically costly, and utilizing the individual
projection operators $(\proj_{C_i})_{i\in I}$ is more tractable.
This issue has given rise to the
advent of {\em splitting}
algorithms, which seek to solve \eqref{e:p} by utilizing operators
associated with the individual sets -- not their intersection.
Projection-based splitting algorithms -- which use the collection
of operators $(\proj_{C_i})_{i\in I}$ instead of
$\proj_{\bigcap_{i\in
I}C_i}$ -- have made previously-intractable
problems of the form \eqref{e:p}
solvable with simpler tools on a larger scale
\cite{Cens15,Comb11,Comb21}. 

While splitting methods have successfully been applied to
projection-based algorithms, relatively little has been done for
the splitting of {\em conditional gradient} (CG / Frank-Wolfe)
algorithms. Standard CG algorithms minimize a smooth function
$f\colon\mathbb{R}^n\to\mathbb{R}$ over one closed
convex constraint set $C\subset\mathbb{R}^n$. While the iterates of
this algorithm do not converge in general \cite{Bolt22}, at
iteration $t\in\NN$,
the average Frank-Wolfe gap (which is closely related to
showing Clarke stationarity~\cite{CGsurvey}) converges at a rate of
$\mathcal{O}(1/\sqrt{t})$, and the primal
gap converges at a rate of $\mathcal{O}(1/t)$ when $f$ is convex
\cite{Pedr20}. A key ingredient of
these algorithms is the {\em linear minimization oracle}, $\lmo_C$,
which computes for a linear objective $c\in\mathbb{R}^n$ a point
in $\Argmin_{x\in C}\langle c\,,\,x\rangle$. Similarly to
traditional projection-based methods, computing $\lmo_{\CCN}$ is
often prohibitively costly, so an algorithm which relies on the
individual operators $(\lmo_{C_i})_{i\in I}$ would be more
tractable.

In principle, if two sets $C_1$ and $C_2$ are polytopes, one could
compute $\lmo_{C_1\cap C_2}$ by
solving a linear program which incorporates the LP formulations of
both $C_1$ and $C_2$. 
However,
since the number of inequalities in an LP formulation can scale
exponentially with dimension \cite{CGsurvey,Roth17}, LPs are
usually only used to implement a polyhedral LMO if there
are no alternatives. In reality, many polyhedra used in applications,
e.g., the Birkhoff polytope and the $\ell_1$ ball, have highly
specialized algorithms for computing their LMO which are faster
than using a linear program \cite{Comb21LMO}. Hence, splitting
algorithms which rely on evaluating the specialized algorithms for
$(\lmo_{C_i})_{i\in I}$ gain the favorable scalability of 
existing LMO implementations.

Conditional gradient methods have seen a resurgence in
popularity since, particularly for high-dimensional settings, LMOs
can be more computationally efficient than projections. For
instance, a common constraint in matrix completion problems is the
spectrahedron
\begin{equation}
\label{e:spec}
S=\{x\in\mathbb{S}^n_{+}\,|\,\text{Trace}(x)=1\},
\end{equation}
where $\mathbb{S}_{+}^n$ is the set of positive semidefinite
$n\times n$ matrices. Evaluating $\proj_{S}$ requires a full
eigendecomposition, while computing $\lmo_S$ only requires
determining a dominant eigenpair \cite{Garb21}. Clearly, there are
high-dimensional settings where evaluating $\lmo_S$ is possible
while $\proj_S$ is too costly \cite{Comb21LMO}. Thus, we are particularly motivated
by high-dimensional
problems in data science (e.g., cluster analysis, graph
refinement, and matrix decomposition) with these LMO-advantaged
constraints, e.g., the nuclear norm ball, the
Birkhoff polytope of doubly stochastic matrices, and the $\ell_1$ ball
\cite{CGsurvey,Ding22,Garb21,Gide18,Rich12,Yang16,Zhan22}.

\begin{example}[Sparse-and-low-rank decomposition]
Let $\tau_1,\tau_2 >0$, and consider the setting when
$\HH=\mathbb{R}^{n\times p}$, $m=2$, 
$C_1$ is the $\ell_1$ ball of radius $\tau_1$, and $C_2$ is the
nuclear norm ball of radius $\tau_2$. This intersection describes a
convex approximation of a set of simultaneously sparse and low-rank
matrices, and it is used for covariance estimation, graph
denoising, and link prediction (e.g., for protein interactions)
\cite{Rich12,Gide18}.
\end{example}

Inexact proximal splitting methods are a natural choice for
solving
\eqref{e:p} in our computational setting, since LMO-based subroutines
can approximate a projection. In the convex case, this approach
appears in \cite{HeHa15,Kolm21,LiuL19,Mill21}. However,
there is often no bound on the number of LMO calls required to meet
the relative error tolerance required of the subroutine, e.g., 
in \cite{Mill21}. Methods which require increasingly-accurate
approximations can drive the number of LMO calls in each subroutine
to infinity
\cite{LiuL19}, and even if a bound on the number of LMO calls
exists, it often depends on the conditioning of the projection
subproblem.

We are interested in algorithms with low iteration complexity,
since they are more tractable on large-scale problems. It
appears that, for this computational setting, the lowest iteration
complexity currently requires one LMO per set per iteration
\cite{Brau22,Gide18,LanR21,MuZh16,Fall19,Yurt19}.
To the best of our knowledge, all algorithms in this class
are restricted to the convex setting. The case when $f=0$
is addressed by \cite{Brau22},
the case when $(C_i)_{i\in I}$ have additional structure is
addressed in \cite{LanR21}, and a matrix recovery problem is
addressed in \cite{MuZh16}. The approaches in
\cite{Gide18,Fall19,Yurt19} essentially show that one CG step is
a sufficient subroutine for an inexact augmented Lagrangian (AL)
approach. These
methods prove convergence of different optimality
criteria at various rates, e.g., arbitrarily close to
$\mathcal{O}(t^{-1/3})$
\cite{Fall19}, $\mathcal{O}(1/\sqrt{t})$ \cite{Yurt19,LanR21}, and
(under restrictions on $m$ or $(C_i)_{i\in I}$) $\mathcal{O}(1/t)$
\cite{Brau22,Gide18,MuZh16}. All of these methods, similarly to
many projection-based splitting algorithms, achieve
approximate feasibility in the sense that a point in the
intersection $\CCN$ is only found asymptotically. 

Our contributions are as follows. We propose a new algorithm in
this class for solving \eqref{e:p} which requires one LMO per set
per iteration. Our algorithm generalizes the vanilla CG algorithm
in the sense that, when $m=1$, both algorithms are identical. It
appears that our algorithm is the first in this
class possessing convergence guarantees for solving \eqref{e:p} in
the setting when $f$ is nonconvex. As is standard in the CG
literature, we
analyze convergence of the average of Frank-Wolfe gaps, and we
prove a rate of $\mathcal{O}(\ln t/\sqrt{t})$ -- only a log factor
slower than the rate for nonconvex CG over a single constraint
($m=1$)
\cite{Pedr20}. We also prove primal gap convergence for the
convex case. Our theory deviates from the AL approach and shows
convergence with direct CG analysis, without imposing additional
structure on our
problem. By recasting \eqref{e:p} in a product space, we derive a
penalized relaxation which is tractable with the vanilla CG
algorithm. The convergence rates in our analysis pertain to our
penalized function, which includes both primal function value and
feasibility terms. At each step of our algorithm, we perform one
vanilla CG step on our product space relaxation; then, we update
the objective
function via a penalty. We provide an analytical and geometric
exploration about the properties of this subproblem as its penalty
changes, as well as its relationships to \eqref{e:p} and related
optimization problems. In particular, we show that for any sequence
of penalty parameters which approach infinity, our subproblems
converge (in several ways) to the original problem.

Our method combines two classical tools from optimization: a
product-space reformulation and a penalty method. Penalty methods
with CG-based subroutines received some attention
several decades ago \cite{Chry97,Migd94}. Our algorithm is related
the Regularized Frank-Wolfe algorithm of \cite{Migd94}, however
their requirements do not apply in our setting. Our algorithm can
also be viewed as an analogue of
\cite{Yurt18}, with different parameters and analysis for 
nonconvex splitting problems \eqref{e:p}.

In the remainder of this section, we introduce notation,
background, and standing assumptions. In Section~\ref{sec:2}, we
demonstrate our product space approach and we establish analytical
results. In Section~\ref{sec:3}, we introduce our algorithm and
prove it converges.

\subsection{Notation, standing assumptions, and auxiliary results}
\label{sec:bg}

Let $\HH$ be a real Hilbert space with inner product
$\scal{\cdot}{\cdot}$, norm $\|\cdot\|=\sqrt{\scal{\cdot}{\cdot}}$,
and identity operator $\Id$. A closed ball
centered at $x\in\HH$ of radius $\varepsilon>0$ is denoted
$B(x;\varepsilon)$. Let $m\in\NN$, set $I=\{1,\ldots,m\}$, and
let $\{\omega_i\}_{i\in I}\subset\left]0,1\right]$ satisfy
$\sum_{i\in I}\omega_i=1$ (e.g., $\omega_i\equiv 1/m$).
\begin{equation}
\label{e:H}
\text{$\HHH=\HH^m$ is the real Hilbert space with inner product
$\scal{\cdot}{\cdot}_{\HHH}= \sum_{i\in
I}\omega_i\scal{\cdot}{\cdot}_{\HH}$,}
\end{equation}
which gives rise to the norm
$\|\cdot\|_{\HHH}=\sqrt{\scal{\cdot}{\cdot}_{\HHH}}$.
Points in $\HHH$ and their
subcomponents are denoted $\bx=(\bx^1,\bx^2,\ldots,\bx^m)$. We call
$\DD=\menge{\bx\in\HHH}{\bx^1=\bx^2=\ldots=\bx^m}$ the {\em diagonal
subspace} of $\HHH$. The block averaging operation and its adjoint
are 
\begin{equation}
\label{e:A}
A\colon\HHH\to\HH\colon \bx\mapsto\sum_{i\in I}\omega_i\bx^i\quad
\text{and}\quad
A^*\colon\HH\to\HHH\colon x\mapsto (x, \ldots, x).
\end{equation}
The {\em projection} operator onto a closed convex set
$C\subset\HH$ is denoted
$\proj_C\colon\HH\to\HH\colon x\mapsto\Argmin_{c\in C}\|x-c\|$.
The {\em distance} and {\em indicator} functions of the set $\DD$
are denoted $\dist_{\DD}\colon\HHH\to\mathbb{R}\colon\bx\mapsto\inf_{\bz\in
\DD}\|\bx-\bz\|$ and
\begin{equation}
\iota_{\DD}\colon\HHH
\to\left[0,+\infty\right]\colon\bx\mapsto\begin{cases}
0 &\text{if}\;\bx\in\DD\\
+\infty &\text{if}\;\bx\not\in\DD,
\end{cases}
\end{equation}
respectively. Note that
\cite[Cor.~12.31]{Livre1}
$\nabla\dist_{\DD}^2/2=
\Id-\proj_{\DD}$, and via \eqref{e:H},
\begin{equation}
\label{e:props}
\frac{1}{2}\dist_{\DD}^2(\bx)=\frac{1}{2}\sum_{i\in
I}\omega_i\|A\bx-\bx^i\|^2,\quad\text{so}\quad
A^*A=\Id-\nabla\left(\frac{1}{2}\dist_{\DD}^2\right)=\proj_{\DD}
\end{equation}
and $\|A\|\leq 1$ \cite[Sec.~4]{Livre1}.
Unless otherwise stated, let $(C_i)_{i\in I}$ be a collection of
nonempty compact convex subsets of $\HH$, let
$L_f>0$, and let $f\colon\HH\to\RR$ be a
Fr\'echet differentiable function which is $L_f$-{\em smooth},
\begin{equation}
\label{e:2}
(\forall (x,y)\in\HH^2)\quad f(y)-f(x)\leq\scal{\nabla f(x)}{y-x}+
\dfrac{L_f}{2}\|y-x\|^2
\end{equation}
and, when restricted to Section~\ref{sec:cvx}, also {\em convex},
\begin{equation}
\label{e:1}
(\forall (x,y)\in\HH^2)\quad \scal{\nabla f(x)}{y-x}\leq f(y)-f(x).
\end{equation}
Technically \eqref{e:2} assumes smoothness of $f$ on $\HH$,
although this work only requires \eqref{e:2} to hold on the
Minkowski sum $\sum_{i\in I}\omega_i C_i$. This assumption
excludes the use of functions only defined on the interior, e.g.,
some logarithmic barriers.

\begin{fact}
\label{f:gsmooth}
Since $\nabla \dist_{\DD}^2/2=\Id-\proj_{\DD}=\proj_{\DD^\perp}$ is
a projection operator onto a nonempty closed convex set, it is
$1$-Lipschitz continuous and therefore $\dist_{\DD}^2/2$ is
$1$-smooth \cite[Corollary~12.31, Section~4]{Livre1}.
\end{fact}

For every $i\in I$ and every $x\in\HH$, the operation $\lmo_i$
returns a point in $\Argmin_{z\in C_i}\scal{x}{z}$. 
The {\em Frank-Wolfe gap} (\FW gap) of $f$ over a compact convex
set $C\subset\HH$ at $x\in\HH$ is 
\begin{equation}
G_{f,C}(x):=\sup_{v\in
C}\scal{\nabla f(x)}{x-v}=\scal{\nabla f(x)}{x-\lmo_{C}(\nabla
f(x))}.
\end{equation}
 Note that, for every $x\in\HH$ \cite{CGsurvey},
\begin{equation}
\label{e:opt}
x\text{ is a stationary point of
}\;\;\underset{x\in C}{\textrm{minimize}\,\, f(x)}
\quad\Leftrightarrow\quad
\begin{cases}
x\in C\\
G_{f,C}(x)\leq 0.
\end{cases}
\end{equation}
Note that if $x\in C$, we always have $G_{f,C}(x)\geq 0$.

\begin{lemma}
\label{l:reg}
Let $f$ and $h$ be real-valued functions on a nonempty set
$C\subset\HH$, let $\lambda,\Delta>0$, and suppose that
\begin{equation*}
x\in\underset{x\in C}{\Argmin}f(x)+\lambda h(x)\quad
\text{and}\quad
z\in\underset{z\in C}{\Argmin}f(z)+(\lambda+\Delta)h(z).
\end{equation*}
Then
$f(x)\leq f(z)$
and $h(z)\leq h(x)$.
\end{lemma}
\begin{proof}
Since $x$ and $z$ are optimal solutions, we have
$f(x)+\lambda h(x)\leq f(z)+\lambda h(z)$
and
$f(z)+(\lambda+\Delta)h(z)\leq f(x)+(\lambda+\Delta) h(x)$,
so in particular,
\begin{equation}
\label{e:g1}
(\lambda+\Delta)(h(z)-h(x))\leq f(x)-f(z)\leq\lambda(h(z)-h(x)).
\end{equation}
Subtracting $\lambda(h(z)-h(x))$ from \eqref{e:g1} implies that
$h(z)-h(x)\leq 0$ which, in view of \eqref{e:g1}, yields
$f(x)-f(z)\leq 0$.
\end{proof}

We assume the ability to compute $\nabla f$, $(\lmo_i)_{i\in I}$,
and basic linear algebra operations, e.g., those in \eqref{e:A}.
Let $f\colon\HH\to\left]-\infty,+\infty\right]$. The {\em
subdifferential} of $f$ at $x\in\HH$ is given by $\partial
f(x)=\menge{u\in\HH}{(\forall y\in\HH)\quad
f(x)+\scal{u}{y-x}\leq f(y)}$. The {\em epigraph} of
$f$ is $\epi f=\menge{(x,\eta)\in\HH\times\RR}{f(x)\leq \eta}$. The
{\em graph} of an operator $M\colon\HH\to2^\HH$ is $\gr
M=\menge{(x,u)\in\HH^2}{u\in M(x)}$. Some of our analytical results
rely on the theory of convergence of sets and set-valued operators;
for a broad review, see \cite{Rock09}.

\begin{definition}
\label{def:conv}
Let $(C_n)_{n\in\NN}$ be a sequence of subsets of $\RR^n$, and let
$(f_n)_{n\in\NN}$ be functions on $\RR^n$. The {\em
outer limit} and {\em inner limit} of $(C_n)_{n\in\NN}$ are
\begin{equation}
\begin{aligned}
\lim\sup_{n\in\NN}(C_n)_{n\in\NN}&=
\menge{x\in\RR^n}{\lim\sup_{n\to+\infty}\dist_{C_n}(x)=0}\\
\text{and}\quad\lim\inf_{n\in\NN}(C_n)_{n\in\NN}&=
\menge{x\in\RR^n}{\lim\inf_{n\to+\infty}\dist_{C_n}(x)=0},
\end{aligned}
\end{equation}
respectively \cite[Ex.~4.2]{Rock09}.
If both limits exist and coincide, this set is the {\em limit} of
$(C_n)_{n\in\NN}$. The sequence $(f_n)_{n\in\NN}$ {\em converges
epigraphically} to a function $f$ on $\mathbb{R}^n$ if the sequence
of epigraphs $(\epi f_n)_{n\in\NN}$ converge to $\epi f$. The
sequence $(\partial f_n)_{n\in\NN}$ {\em converges graphically} to
$\partial f$ if $(\gr\partial f_n)_{n\in\NN}$ converges to
$\gr\partial f$.
\end{definition}

\section{Splitting constraints with a product space}
\label{sec:2}
This section outlines our algorithm and provides additional
analysis relating our approach to similar problems in optimization.

\subsection{Algorithm design}
\label{sec:algo}

The vanilla conditional gradient algorithm solves
\begin{equation}
\label{e:OG}
\underset{x\in C}{\text{minimize}}\;\;f(x)
\end{equation}
using $\lmo_C$ and gradients of $f$. However, 
one of the central hurdles in designing a tractable CG-based
splitting algorithm is finding a way to enforce membership in the
constraint $\CCN$ without access to its projection or LMO.
Our approach to solving this issue comes from the following 
construction on the product space $\HHH$ (see Section~\ref{sec:bg}
for notation and Fig.~\ref{fig:prod} for visualization).
\begin{proposition}
\label{prop:decomp}
Let $(C_i)_{i\in I}$ be a collection of nonempty 
subsets of $\HH$, and let $\DD\subset\HHH$ denote the diagonal
subspace. Then
\begin{align}
(\forall x\in\HH)\quad 
(x,\ldots,x)\in
\DD\cap\bigtimes_{i\in I}C_i
\quad&\Leftrightarrow\quad x\in\bigcap_{i\in I}C_i \\
(\forall \bx\in\HHH)\quad\quad\bx\in\DD\cap\bigtimes_{i\in I}C_i
\quad&\Leftrightarrow\quad
(\exists\, x\in\HH)\quad
\begin{cases}
\bx=(x,\ldots,x)\\
x\in\bigcap_{i\in I}C_i,
\end{cases}
\end{align}
\end{proposition}
\begin{proof}
Clear from construction.\footnote{The type of construction in
Proposition~\ref{prop:decomp} goes back to the work of Pierra
\cite{Pier84}.}
\end{proof}

\begin{figure}[t]
\centering
\includegraphics[width=9.8cm]{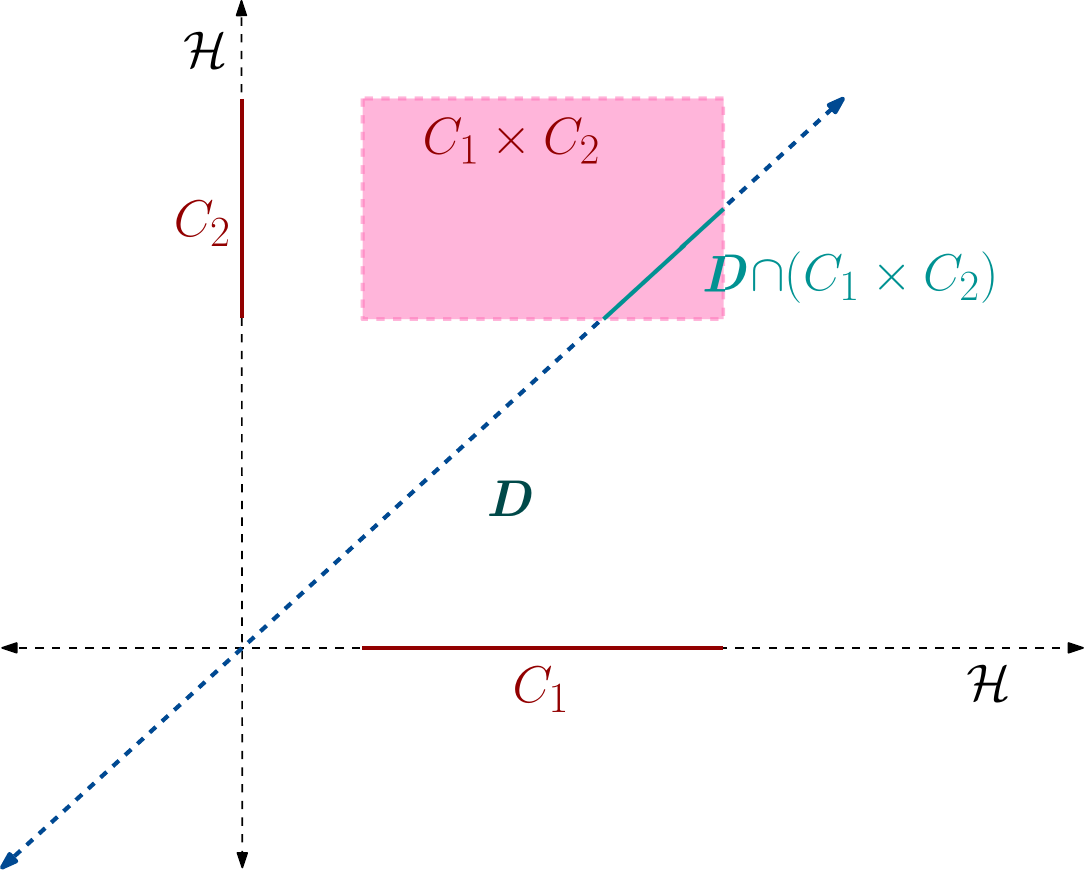}
\caption{
Visualization of the product space for $\HH=\RR$ and $m=2$. Our algorithm
produces iterates $\bx_t$ which are always inside the shaded
constraint set, and their averages $A^*A\bx_t$ are always on the
diagonal subspace $\DD$. The solid segment where $C_1\times C_2$
and $\DD$ intersect corresponds precisely to our split feasibility
constraint via Proposition~\ref{prop:decomp}.
}
\label{fig:prod}
\end{figure}

Proposition~\ref{prop:decomp} provides a decomposition of the
split feasibility constraint $\CCN$ in
terms of two simpler sets $\DD$ and $\bigtimes_{i\in I}C_i$. This
yields a product space reformulation of \eqref{e:p}
\begin{equation}
\label{e:prod}
\underset{\bx\in\CCX}{\text{minimize}}\;\;f(A\bx)+\iota_{\DD}(\bx).
\end{equation} 
The constraints $\DD$ and $\bigtimes_{i\in I}C_i$ are simpler in
the sense
that, even in our restricted computational setting, we can compute
operators to enforce them. In particular, the projection onto
$\DD$ is computed by simply repeating the average of all components
in every component
$\proj_{\DD}\bx=A^*(\sum_{i\in I}\omega_i \bx^i)$. Critically, this
operation is cheap, so one can actually evaluate the gradient 
$\nabla\dist_{\DD}^2/2=\Id-\proj_{D}$ even though it involves a
projection. The constraint $\bigtimes_{i\in I}C_i$ is readily
processed using the following property.

\begin{fact}
\label{f:22}
Let $(C_i)_{i\in I}$ be a collection of nonempty compact convex
subsets of $\HH$. Then
\begin{equation}
(\forall\bx\in\HHH)\quad
\lmo_{(\bigtimes_{i\in
I}C_i)}(\bx)=(\lmo_{C_1}\bx^1,\ldots,\lmo_{C_m}\bx^m).
\end{equation}
\end{fact}
In particular, to evaluate an LMO for the product $\bigtimes_{i\in
I}C_i$, it suffices to evaluate the individual operators  
$(\lmo_i)_{i\in I}$ once.

With these ideas in mind, let us introduce the penalized function
\begin{equation}
\label{e:4}
F_{\lambda}\colon\HHH\to\mathbb{R}\colon\bx\mapsto f(A\bx)+\lambda
\frac{1}{2}\dist_{\DD}^2(\bx),
\end{equation}
which, for every $\lambda\geq 0$, is $(L_f+\lambda)$-smooth (cf.
Fact~\ref{f:gsmooth}).
We observe that for every penalty parameter $\lambda_t\geq 0$, even
under our
restricted computational setting, the following relaxation of
\eqref{e:prod} is still tractable with the vanilla CG algorithm
\begin{equation}
\label{e:49}
\underset{\bx\in\CCX}{\text{minimize}}\;\;F_{\lambda_t}(\bx).
\end{equation}
Indeed, vanilla CG requires the ability to compute the
gradient of the objective function and the LMO of the constraint.
Computing $\nabla F_{\lambda}=\nabla f+\lambda(\Id-\proj_{\DD})$
amounts to one evaluation of $\nabla f$, computing one average, and
some algebraic manipulations. By promoting membership of
$\DD$ via the objective function, 
we are left with the LMO-amenable constraint $\CCX$. 

The core idea of our algorithm is, at each iteration $t\in\NN$, to
perform
one Frank-Wolfe step to the relaxed subproblem \eqref{e:49}.
Then, between iterations, we update the objective function
in \eqref{e:49} via $\lambda_t$ to promote feasibility. Although
\eqref{e:49} is a relaxation of the intractable problem
\eqref{e:prod}, taking $\lambda_t\to\infty$ suffices to show
convergence in \FW gap (and primal gap, in the convex case) to
solutions of \eqref{e:prod} and hence \eqref{e:p}; this is
substantiated in Sections~\ref{sec:interp} and \ref{sec:3}. 
For every $\bx\in\HHH$, the $i$th component of the
gradient is given by $\nabla F_\lambda(\bx)^i=\nabla
f(A\bx)+\lambda(\bx^i-A\bx)$. So, a CG step applied to
\eqref{e:49} yields Algorithm~\ref{alg:scg}. While
Section~\ref{sec:3} contains the precise schedules for
Lines~\ref{a:lam} and \ref{a:step}, the parameters behave like
$(\lambda_t,\gamma_t)=(\mathcal{O}(\ln t),\mathcal{O}(
1/\sqrt{t}))$.

\begin{algorithm}[ht]
\caption{Split conditional gradient (SCG) algorithm}
\label{alg:scg}  
\begin{algorithmic}[1]
\REQUIRE Smooth function $f$, weights $\{\omega_i\}_{i\in
I}\subset\left]0,1\right]$ such that $\sum_{i\in I}\omega_i=1$,
point $\bx_0\in\CCX$ 
\STATE $x_0\leftarrow \sum_{i\in I}\omega_i\bx_0^i$
\FOR{$t=0, 1$ \textbf{to} $\dotsc$}
\STATE Choose penalty parameter $\lambda_t\in
\left]0,+\infty\right[$
\label{a:lam}
\STATE Choose step size $\gamma_t\in \left]0,1\right]$
\label{a:step}
\STATE $g_t\leftarrow \nabla f(x_t)$
\COMMENT{Store $\nabla f(A\bx_t)$ for CG step on \eqref{e:49}}
\FOR{$i=1$ \textbf{to} $m$}
\label{a:block}
\STATE $\bv_{t}^i\leftarrow \lmo_i(g_t+\lambda_t(\bx_{t}^i-x_t))$
\COMMENT{LMO applied to $\nabla F_{\lambda_t}(\bx_t)^i$}
\label{a:lmo}\\
\STATE $\bx_{t+1}^i\leftarrow
\bx_{t}^i+\gamma_t(\bv_{t}^i-\bx_{t}^i)$
\COMMENT{CG step in $i$th component}
\label{a:new}\\
\ENDFOR
\label{a:endblock}
\STATE $x_{t+1}\leftarrow \sum_{i\in I}\omega_i\bx_{t+1}^i$
\COMMENT{Approximate solution by averaging}
\label{a:avg}
\ENDFOR
\end{algorithmic}
\end{algorithm}

CG-based algorithms possess the advantage that, at every iteration,
the iterates are feasible (i.e., for \eqref{e:OG}, $x_t\in C$).
Our approach inherits this familiar property; however, since we
solve a product space relaxation, $\bx_t\in\CCX$ and hence, for
every $i\in I$, the $i$th component of our sequence is feasible for
the $i$th constraint, i.e., $(\bx_{t}^i)_{t\in\NN}\in C_i$.
Importantly, this does not guarantee that any subcomponent
$\bx_{t}^{i}$ resides in $\CCN$, so they are not feasible for the
splitting problem \eqref{e:p}; feasibility in $\CCN$ is
acquired ``in the limit'', by showing that $\bx_t\in\CCX$ and
$\dist_{\DD}(\bx_t)\to 0$ (proven in Section~\ref{sec:3}). 

In practice, one needs a route to construct an approximate
solution to \eqref{e:p} in $\HH$ from an iterate of
Algorithm~\ref{alg:scg} in the product space $\HHH$. Instead of
taking a component, our approximation is the average computed in
Line~\ref{a:avg}, since
\begin{equation}
\label{e:p12}
(\forall \bx\in\HHH)\quad\quad\bx\in\DD\cap\bigtimes_{i\in I}C_i
\quad\Rightarrow\quad
A\bx\in\bigcap_{i\in I}C_i
\end{equation}
is a strict implication.
Hence the condition $A\bx\in\CCN$ is easier to satisfy than
$\bx\in\DD\cap\CCX$ (see also Sec.~\ref{sec:Frel}).

\begin{remark}
\label{r:OG}
If we have only $m=1$ set constraint, then $A=\Id$, and
$\HH=\HHH=\DD$, so at every iteration $t\in\NN$,
$F_{\lambda_t}=f(x)$. Therefore, the classical CG algorithm is a
special case of Algorithm~\ref{alg:scg}.
\end{remark}

\begin{remark}
\label{r:wi}
The convex weights $(\omega_i)_{i\in I}$ in
\eqref{e:H}--\eqref{e:A} can be used to preferentially
promote membership of the approximate solution $A\bx_t$ into some
constraint(s) over others. If all constraints are equally
important, we suggest $\omega_i\equiv 1/m$; if one constraint $C_j$
is more important, then by
selecting $\omega_j>\omega_i$ for all $i\in I\setminus\{j\}$, 
the weighted average $A\bx_t$ is closer (in a Euclidean sense) to
$\bx_t^j\in C_j$ than the components of $\bx_t$ in other sets
$(C_i)_{i\in I\setminus \{j\}}$.
\end{remark}

\begin{fact}
All gradients are computed with respect to the inner product on
$\HHH$. If a gradient with respect
to the Euclidean inner product
$\scal{\cdot}{\cdot}_{\mathcal{E}}\coloneqq\sum_{i\in
I}\scal{\cdot}{\cdot}_{\HH_i}$ is provided, then one may construct the
gradient with respect to the Hilbertian inner product \eqref{e:H}
on $\HHH$ by multiplying the $i$th component by $\omega_i$ for all
$i\in I$.
\end{fact}

\subsection{Analysis}
Here we gather analytical results pertaining to our algorithm, 
the geometry of our product-space construction, and how our relaxed
problem relates to other classical problems in optimization. While
these results are interesting in their own right, many are also
used to show convergence in Section~\ref{sec:3}.

\subsubsection{Geometry (and tractability) of penalty functions on
the Cartesian product}
\label{sec:Frel}

As seen in Section~\ref{sec:algo}, Algorithm~\ref{alg:scg}
promotes split feasibility by, at every iteration $t\in\NN$,
requiring that $\bx_t\in\CCX$ and penalizing the distance from
$\bx_t$ to $\DD$. However, as seen in \eqref{e:p12},
$\dist_{\DD}(\bx_t)=0$ is a sufficient (but not
necessary) condition to acquire a feasible average $A\bx_t\in\CCN$;
see Fig.~\ref{fig:prod3}.
In this section, we present a penalty function which precisely
characterizes this condition. Via a simple geometric argument based
on the projection theorem, we guarantee that although
utilizing this penalty is not computationally tractable, it is
nonetheless minimized when $\dist_{\DD}$ vanishes. These results
also further substantiate the claim that $\bx\in\DD\cap\CCX$ is a
stricter condition than $A\bx\in\CCN$, which is our motivation to 
use the average in Line~\ref{a:avg} of Algorithm~\ref{alg:scg} as
our approximate solution to \eqref{e:p}.

\begin{proposition}
\label{p:2}
Let $(C_i)_{i\in I}$ be a collection of nonempty closed convex
subsets of $\HH$, let $\DD\subset\HHH$ denote the diagonal
subspace, and set
\begin{equation}
\label{e:d}
d\colon\HHH\to\left]-\infty,+\infty\right]\colon
\bx\mapsto\sum_{i\in I}\omega_i\dist^2_{C_i}(A\bx).
\end{equation}
Then, for every $\bx\in\HHH$, the following are equivalent.
\begin{enumerate}
\item
\label{p:2a}
$d(\bx)=0$.
\item
\label{p:2b}
$A\bx\in\CCN$.
\item
\label{p:2c}
$\proj_{\DD}(\bx)\in\CCX$.
\end{enumerate}
\end{proposition}
\begin{proof}
\ref{p:2a}$\Rightarrow$\ref{p:2b}: For every $i\in I$,
$0\leq \omega_i\dist^2_{C_i}(A\bx)\leq d(\bx)=0$. Since
$\omega_i>0$, it follows that $\dist^2_{C_i}(A\bx)=0$ and hence $A\bx\in
C_i$.

\ref{p:2b}$\Rightarrow$\ref{p:2c}: By applying $A^*$ to the
inclusion \ref{p:2b},
\eqref{e:props} implies that
\begin{equation}
\proj_{\DD}\bx=A^*A\bx
\in A^*\CCN
=\Menge{(x,\ldots,x)\in\HHH}{x\in\CCN}.
\end{equation}
So, by Proposition~\ref{prop:decomp}, 
$\proj_{\DD}\bx\in\Menge{\bx\in\HHH}{\bx\in\DD\cap\CCX}
\subset\CCX$.

\ref{p:2c}$\Rightarrow$\ref{p:2a}: 
We begin by observing that 
\begin{equation}
\proj_{\CCX}(\proj_{\DD}\bx)=\underset{\bc\in\CCX}{\Argmin}\|\bc-A^*A\bx\|_{\HHH}^2
=\underset{\bc\in\CCX}{\Argmin}\sum_{i\in
I}\omega_i\|\bc^i-A\bx\|_{\HH}^2
\end{equation}
is a separable problem whose solution is
$\left(\proj_{C_1}(A\bx),\ldots,\proj_{C_m}(A\bx)\right)$.
Therefore, 
\begin{align}
\label{e:24}
d(\bx)
=\sum_{i\in I}\omega_i\|A\bx-\proj_{C_i}(A\bx)\|_{\HH}^2
=\|\proj_{\DD}\bx-\proj_{\CCX}(\proj_{\DD}\bx)\|_{\HHH}^2.
\end{align}
Since $\proj_{\DD}(\bx)=\proj_{\CCX}(\proj_{\DD}\bx)$, we conclude
$d(\bx)=0$.
\end{proof}

\begin{figure}[H]
\centering
\includegraphics[width=8.8cm,trim={5cm 8.2cm 2.5cm
1.2cm},clip]{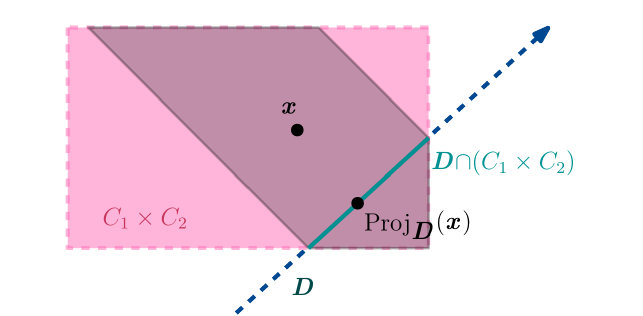}
\caption{
Zoomed view of Fig.~\ref{fig:prod}.
The darker shaded area is the collection of points $\bx\in\CCX$
for which $\proj_{\DD}(\bx)$ remains in $\CCX$. By
Proposition~\ref{p:2}, this is the set of points satisfying
$A\bx\in\CCN$. This exemplifies that the implication \eqref{e:p12}
is strict.
}
\label{fig:prod3}
\end{figure}

Since we do not assume the ability to project onto the sets
$(C_i)_{i\in I}$, evaluating $\nabla d=2\sum_{i\in
I}\omega_i(\Id-\proj_{C_i})$ is not possible. Therefore, replacing
$F_{\lambda}$ in \eqref{e:49} with the composite function
$f(A\bx)+\lambda d(\bx)$ is not tractable with a vanilla CG-based
approach. However, $d$ is closely related to our penalty
function $\dist^2_{\DD}$ via the following result.

\begin{corollary}
\label{c:dg}
In the setting of Proposition~\ref{p:2}, let $\bx\in\bigtimes_{i\in
I}C_i$, set $\by=\proj_{\DD}\bx$ and set
$\bp=\proj_{(\bigtimes_{i\in I} C_i)}(\by)$. Then
\begin{equation}
\label{e:dist-g}
d(\bx)
=\dist_{\DD}^2(\bx)-\|\bx-\bp\|^2 +2
\underbrace{\scal{\bx-\bp}{\by-
\bp}}_{\leq 0}.
\end{equation}
In consequence, $0\leq d(\bx)\leq
\dist_{\DD}^2(\bx)$.
\end{corollary}
\begin{proof}
Follows from Lemma~2.12 and Theorem~3.16 of \cite{Livre1}.
\end{proof}

Since the iterates of Algorithm~\ref{alg:scg} always reside in
$\CCX$, Corollary~\ref{c:dg} reinforces our choice of $A\bx_t$ as
our approximate solution of \eqref{e:p}. Firstly, its implication
that
$\dist^2_{\DD}(\bx)=0\Rightarrow d(\bx)=0$ underlines the
observation from \eqref{e:p12} that $A\bx_t\in\CCN$ is easier to
satisfy than $\bx\in\DD\cap\CCX$. Furthermore, by characterizing
the gap between $d$ and $\dist_{\DD}^2$, we see that there are
plenty of points for which the inequality between $d$ and
$\dist_{\DD}^2$ is strict, e.g., those $\bx\in\CCX$ for which
$\bx\neq\bp$ (see also Fig.~\ref{fig:prod3}). Due to this
strictness, $d(\bx_t)$ may vanish far {\em before}
$\dist_{\DD}^2(\bx_t)$ vanishes over the iterations of
Algorithm~\ref{alg:scg}.
This is consistent with our preliminary numerical observations
that $A\bx_t\in\CCN$ often occurs before
$\dist^2_{\DD}(\bx_t)$ vanishes.

\begin{remark}
\label{rmk:g}
Another natural penalty to consider is
\begin{equation}
g\colon\bx\mapsto\dist_{\CCN}^2(A\bx)=
\|\proj_{\DD}\bx-\proj_{\DD\cap\CCX}\bx\|^2,
\end{equation}
although evaluating $\nabla
g=2A^*(\Id-\proj_{\CCN})(A\bx)$ involves computing an
intractable projection. While, for every $\bx\in\HH$, $g$ and $d$
(see \eqref{e:d}) have the order
\begin{equation}
d(\bx)=\sum_{i\in I}\omega_i \inf_{c\in C_i}\|A\bx-c\|^2\leq
\sum_{i\in I}\omega_i \inf_{c\in \bigcap_{i\in I}C_i}\|A\bx-c\|^2
=g(\bx),
\end{equation}
there is no general ordering between $g$ and our penalty
$\dist_{\DD}^2$ for $\dim(\HH)\geq 2$. However, using
\eqref{e:H}--\eqref{e:A} reveals that they are related in the
following geometric sense
\begin{equation}
\label{e:gbound}
\begin{multlined}
\sum_{i\in
I}\omega_i\|\bx^i-\proj_{\bigcap_{i\in
I}C_i}(A\bx)\|^2=
\|\bx-A^*A\bx+A^*A\bx-A^*\proj_{\CCN}(A\bx)\|^2\\
=g(\bx)+\dist_{\DD}^2(\bx)-
\underbrace{2\scal{A^*\proj_{\bigcap_{i\in
I}C_i}(A\bx)-A^*A\bx}{\bx-A^*A\bx}}_{=0}.
\end{multlined}
\end{equation}
Since the lefthand and righthand vectors in the scalar product are
in $\DD$ and $\DD^\perp$ respectively, $g$ and $\dist_{\DD}^2$
describe the squared magnitude of two orthogonal vectors. 
\end{remark}

\begin{figure}[t]
\centering
\includegraphics[width=10.5cm,trim={6cm 7.0cm 3.5cm
3.0cm},clip]{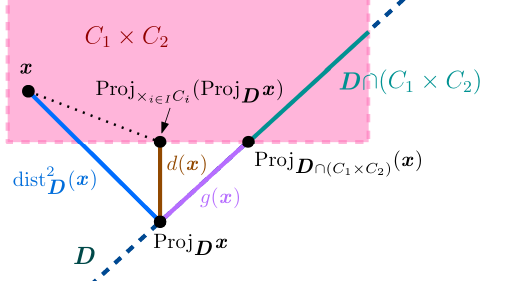}
\caption{
Zoomed view of Fig.~\ref{fig:prod} displaying
geometric relationships between $\dist_{\DD}^2$, $d$,
and $g$ (see \eqref{e:24} and Remark~\ref{rmk:g}).
This shows a feasible point $\bx\in\CCX$ and several vectors whose
squared magnitude are equal to the given labels.
Corollary~\ref{c:dg} describes the relative magnitude of
$\dist^2_{\DD}$ and $d$, as well as the obtuse angle between $\bx$,
$\proj_{\CCX}(\proj_{\DD}\bx)$, and $\proj_{\DD}\bx$.
Remark~\ref{rmk:g} describes the orthogonality seen in the angle
between $\bx$, $\proj_{\DD}\bx$, and $\proj_{\DD\cap\CCX}(\bx)$.
We also see
$\proj_{\DD\cap\CCX}\bx=\proj_{\DD\cap\CCX}(\proj_{\DD}\bx)$, which
holds in general
\cite[Prop.~24.18]{Livre1}.
}
\label{fig:prod2}
\end{figure}

\subsubsection{Interpolating constraints: From the Minkowski sum to
the intersection}
\label{sec:interp}

This section presents an analysis of how our subproblem
\eqref{e:49} changes with the parameter $\lambda$. In addition to
their utility in Section~\ref{sec:3} to prove that our sequence of
relaxations \eqref{e:49} actually solves the correct problem
\eqref{e:prod}, the results in this section show that \eqref{e:49}
connects two classical problems in optimization.

From a certain perspective, \eqref{e:49} ``interpolates'' from the
following problem (when $\lambda=0$) over the Minkowski sum
\begin{equation}
\label{e:mink}
\underset{x\in\CCM}{\text{minimize}}\;\;f(x),\quad\text{where}\quad
\CCM=\Menge{\sum_{i\in I}\omega_ic^i}{(\forall i\in I)\;\;
c^i\in C_i},
\end{equation}
to the splitting problem \eqref{e:p} (when
$\lambda\nearrow+\infty$). We shall make this latter observation
precise via several notions of convergence in
Proposition~\ref{p:conv}.

\begin{remark}
While this article is predominantly focused on \eqref{e:p}, it is
worth noting that, when $\lambda=0$, the problems
\eqref{e:49} and \eqref{e:mink} coincide in the sense that, for
every solution $\bx^*$ of \eqref{e:49}, $A\bx^*$ solves
\eqref{e:mink} (and for every solution $\sum_{i\in I}\omega_ix^i$
of \eqref{e:mink}, $(x^i)_{i\in I}$ solves \eqref{e:49}).
Therefore, Fact~\ref{f:22} leads to a Frank-Wolfe approach
to solving \eqref{e:mink}. The Minkowski sum constraint arises
in Bayesian learning, placement problems, and robot motion planning
\cite{Bern09,Duan19,WonX19,Loza79}.
\end{remark}

We begin with the following observations about how $F_{\lambda}$
relates as $\lambda$ varies.
\begin{lemma}
\label{l:Frel}
Let $f\colon\HH\to\RR$, let $\lambda,\Delta \in\RR$, let
$\DD\subset\HHH$ be nonempty, and set $F_{\lambda}\colon\bx\mapsto
f(A\bx)+\lambda\dist_{\DD}^2(\bx)/2$. Then,
\begin{equation}
\label{e:Frel}
(\forall \bx\in\HHH)\quad F_{\lambda}(\bx)=F_{\lambda+\Delta}(\bx)-
\Delta\frac{1}{2}\dist_{\DD}^2(\bx).
\end{equation}
In consequence, if $\Delta\geq 0$, then $F_{\lambda}(\bx)\leq
F_{\lambda+\Delta}(\bx)$ and
$$\inf F_{\lambda}(\CCX)\leq\inf F_{\lambda+\Delta}(\CCX).$$
\end{lemma}
\begin{proof}
$F_{\lambda}(\bx)=f(A\bx)+(\lambda+\Delta)\dist_{\DD}^2(\bx)/2-\Delta\dist_{\DD}^2(\bx)/2=$

\noindent $F_{\lambda+\Delta}(\bx)-\Delta\dist_{\DD}^2(\bx)/2.$
\end{proof}

Next, we show that the optimal value of \eqref{e:49} is sandwiched
between that of the splitting problem \eqref{e:p} and the Minkowski
sum problem \eqref{e:mink}.

\begin{proposition}
\label{p:8}
Let $f\colon\HH\to\RR$, let $\lambda\geq 0$, let $\DD\subset\HHH$
be nonempty, set $F_{\lambda}\colon\bx\mapsto
f(A\bx)+\lambda\dist_{\DD}^2(\bx)/2$, and let $(C_i)_{i\in I}$ be a
collection of nonempty compact convex subsets of $\HH$ such that
$\CCN\neq\varnothing$. Then
\begin{align}
\label{e:F1}
\inf_{x\in \bigcap_{i\in I} C_i} f(x) \geq
\inf_{\bx\in\CCX}
F_{\lambda}
\geq \inf_{x\in\CCM} f(x).
\end{align}
\end{proposition}
\begin{proof}
To show the first inequality, we note that
for every $\bx\in\DD$, $\dist_{\DD}^2(\bx)=0$, so using the
product space formulation \eqref{e:prod} of \eqref{e:p},
\begin{equation}
\inf_{x\in\bigcap_{i\in I} C_i}
f(x)=\inf_{\bx\in\DD\cap\CCX}f(A\bx)+\frac{\lambda}{2} \dist_{\DD}^2(\bx)
\geq\inf_{\bx\in\CCX}F_{\lambda}(\bx).
\end{equation}
The second inequality follows from the observation that
\eqref{e:mink} coincides with \newline
$\inf_{\bx\in\CCX}F_{0}(\bx)$, so by
Lemma~\ref{l:Frel} we have
$\inf_{\bx\in\CCX}F_{\lambda}(\bx)\geq
\inf_{\bx\in\CCX}F_{0}(\bx)$.
\end{proof}

It turns out that, for an increasing sequence
of penalty parameters $(\lambda_n)_{n\in\NN}$, the ordering of
Proposition~\ref{p:8} is preserved if we only consider the optimal
values of $f$ (instead of $F_{\lambda}$). Intuitively, the order is
reversed when we compare optimal values of the penalty
$\dist_{\DD}^2$.

\begin{corollary}
Let $f\colon\HH\to\RR$, let $\DD\subset\HHH$ be nonempty, set
$F_{\lambda}\colon\bx\mapsto f(A\bx)+\lambda\dist_{\DD}^2(\bx)/2$,
and let $(C_i)_{i\in I}$ be a collection of nonempty compact convex
subsets of $\HH$ such that $\CCN\neq\varnothing$. 
Suppose that $(\lambda_t)_{t\in\NN}$ is an increasing
sequence of nonnegative real numbers and,
for every $t\in\NN$, let $\bx_t^*$ be a minimizer of
$F_{\lambda_t}$ over $\CCX$. Then
\begin{equation}
\inf_{x\in \bigcap_{i\in I} C_i} f(x) \geq
f\left(A\bx_{t+1}^*\right)\geq
f\left(A\bx_{t}^*\right)\geq
\inf_{x\in\CCM} f\left(x\right).
\end{equation}
If $\bz\in\Argmin_{\bx\in\bigtimes_{i\in I} C_i}
f\left(A\bx\right)$ (i.e., $A\bz$ solves \eqref{e:mink}), then
\begin{equation}
0\leq \dist_{\DD}^2(\bx_{t+1}^*)\leq \dist_{\DD}^2(\bx_t^*)\leq
\dist_{\DD}^2(\bz).
\end{equation}
\end{corollary}
\begin{proof}
Follows from Lemma~\ref{l:reg} and Proposition~\ref{p:8}.
\end{proof}

The following example demonstrates that the penalty sequence
$(\lambda_t)_{t\in\NN}$ may need to tend to $+\infty$ in order 
for the solutions of \eqref{e:49} and \eqref{e:p} to coincide.
\begin{example}
\label{ex:linf}
Set $\HH=\RR$, set $f=\|x\|^2/2$, let $z\geq0$, set $C_1=\{z\}$,
and set $C_2=[-z-1,z+1]$. Clearly, $z=\Argmin_{x\in C_1\cap
C_2}f(x)$. However,
it is straightforward to verify that, for every $\lambda\geq 0$,
$\bx^*_{\lambda}=((\lambda-1)z/(1+\lambda),z)$ is
the unique minimizer of $F_{\lambda}$ over $C_1\times C_2$.
Since $A\bx^*_{\lambda}=\lambda z/(1+\lambda)\neq z$, the solutions
of \eqref{e:49} and \eqref{e:p} (via \eqref{e:prod}) do not
coincide for finite $\lambda$; taking $\lambda\to+\infty$ implies
$A\bx^*_{\lambda}\to z$.
\end{example}

The following result establishes three notions of convergence
(see Definition~\ref{def:conv})
relating the problems \eqref{e:49} and \eqref{e:p} (via its
equivalent product space formulation \eqref{e:prod}). For this
result, we rely on the fact that every constrained optimization
problem can be described using a single objective function via the
use of indicator functions.

\begin{proposition}
\label{p:conv}
Let $f\colon\HH\to\RR$, let $(C_i)_{i\in I}$ be a collection of
nonempty compact convex subsets of $\HH$ such that $\bigcap_{i\in
I}C_i\neq\varnothing$, and let $\DD$ denote the diagonal subspace
of $\HHH$. Suppose that $(\lambda_t)_{t\in\NN}\to+\infty$ and, for
every $t\in\NN$, set $\boldsymbol{f}_t= f\circ
A+\lambda_t\dist_{\DD}^2/2+\iota_{\bigtimes_{i\in I}C_i}$. Then the
following hold.
\begin{enumerate}
\item
\label{p:convp}
$\boldsymbol{f}_t$ converges pointwise to $f\circ A+\iota_{\DD\cap\CCX}$.
\item
\label{p:conve}
Suppose $\HH=\mathbb{R}^n$. Then
$\boldsymbol{f}_t$ converges epigraphically to $f\circ
A+\iota_{\DD\cap\CCX}$.
\item
\label{p:convg}
Suppose $\HH=\mathbb{R}^n$ and $f$ is convex. Then
$\partial \boldsymbol{f}_n$ converges graphically to $\partial
(f\circ A+\iota_{\DD\cap\CCX})$.
\end{enumerate}
\end{proposition}
\begin{proof}
Since 
\begin{equation}
\iota_{\bigtimes_{i\in I}C_i}+\iota_{D}=\iota_{D\cap\bigtimes_{i\in
I}C_i},
\end{equation}
it suffices to show that $\lambda_t\dist_{\DD}^2/2$ converges to
$\iota_{D}$ under each notion of convergence. \newline
\ref{p:convp}: Let $\bx\in\HHH$. If $\bx\in\DD$, then for every
$n\in\NN$,
$\lambda_t\dist_{\DD}^2(\bx)/2=0=\iota_{\DD}(\bx)$. On the other
hand, if $\bx\not\in\DD$, then $0<\lambda_t\dist_{\DD}^2(\bx)/2\to
+\infty=\iota_{D}(\bx)$.
\newline
\ref{p:conve}:
Let $\bx\in\HHH$.
By \cite[Proposition~7.2]{Rock09},
it suffices to show both of the following.
\begin{align}
\label{e:conve1}
\text{For some sequence}\;(\bx_t)_{t\in\NN}\text{ converging to }&
\bx,\quad \lim\sup_{t\in\NN}\frac{\lambda_t}{2}\dist_{\DD}^2(\bx_t)\leq
\iota_{\DD}(\bx).\\
\label{e:conve2}
\text{For every sequence}\;(\bx_t)_{t\in\NN}\text{ converging to }&
\bx,\quad\lim\inf_{t\in\NN}\frac{\lambda_t}{2}\dist_{\DD}^2(\bx_t)\geq
\iota_{\DD}(\bx).
\end{align}
To realize \eqref{e:conve1}, we consider the constant sequence
$(\bx_t)_{t\in\NN}\equiv\bx$. By \ref{p:convp},
\begin{equation}
\lim\sup_{t\in\NN}\lambda_t\dist_{\DD}^2(\bx_t)/2=
\lim_{t\in\NN}\lambda_t\dist_{\DD}^2(\bx)/2=\iota_{\DD}(\bx),
\end{equation}
so this is always satisfied with equality. To show
\eqref{e:conve2}, let $(\bx_t)_{t\in\NN}$ be a sequence converging to
$\bx$. If $\bx\in\DD$, then since $\dist_{\DD}^2\geq 0$ and
$\iota_{\DD}(\bx)=0$, \eqref{e:conve2} holds.
Otherwise, if $\bx\not\in\DD$, then
there exists a radius $\varepsilon>0$ such that $B(\bx;\varepsilon)\cap
\DD=\varnothing$. Since $\dist_{\DD}^2$ is continuous and only
vanishes on $\DD$,
we know $\eta:=\inf_{\by\in
B(\bx;\varepsilon/2)}\dist_{\DD}^2(\by)/2>0$. Therefore,
since $\bx_t\to\bx$, we have that, for some $N\in\NN$, $n>N$ implies
that $\bx_t\in B(\bx;\varepsilon/2)$, hence
\begin{equation}
\frac{\lambda_t}{2} \dist_{\DD}^2(\bx_t)\geq\lambda_t\eta\to +\infty.
\end{equation}
In particular, $\lim_{t\in\NN}\lambda_t
\dist_{\DD}^2(\bx_t)/2=+\infty=\iota_{\DD}(\bx)$ so we are done.
\newline
\ref{p:convg}:
Follows from \ref{p:conve} and Attouch's Theorem
\cite[Theorem~12.35]{Rock09}.
\end{proof}

In general, the functions in Proposition~\ref{p:conv} do not
converge uniformly\footnote{Uniform convergence for extended-real
valued functions is defined in \cite{Rock09}.}.
In spite of this, it turns out that one can nonetheless commute the
limit with an infimum, hence showing that the optimal
values of our subproblems \eqref{e:49} converge to the optimal
value of \eqref{e:p}.

\begin{proposition}
\label{p:fconv}
Let $f\colon\HH\to\RR$, let $(C_i)_{i\in I}$ be a
collection of nonempty compact convex subsets of $\HH$, let
$\DD$ denote the diagonal subspace of $\HHH$, and for every
$\lambda\geq 0$, set $F_{\lambda}\colon\bx\mapsto
f(A\bx)+\lambda\dist_{\DD}^2(\bx)/2$. Suppose that
$(\lambda_n)_{n\in\NN}\nearrow +\infty$. Then
\begin{equation}
\label{e:308}
\lim_{t\to+\infty}\left(\inf_{\bx\in\CCX}F_{\lambda_t}(\bx)\right)
\to
\inf_{\bx\in\CCX}\left(\lim_{t\to\infty}F_{\lambda_t}(\bx)\right)
=\inf_{x\in\CCN}f(x).
\end{equation}
\end{proposition}
\begin{proof}
First, we point out that the equality in \eqref{e:308} follows from
Proposition~\ref{p:conv} and the fact that the minimal values of
\eqref{e:p} and \eqref{e:prod} coincide. Let
$$\mu<\inf_{x\in\CCN}f(x)
=\inf_{\bx\in\CCX}f(A\bx)+\iota_{\DD}(\bx).$$
By Proposition~\ref{p:conv}, for every $\bx\in\CCX$,
$\lim_{t\to\infty}F_{\lambda_t}(\bx)=f(A\bx)+\iota_{\DD}(\bx)>\mu$.
Since $\CCX$ is compact, for $t\in\NN$ sufficiently large,
$\inf_{\bx\in\CCX} F_{\lambda_t}(\bx)\geq\mu$, which implies
(via Proposition~\ref{p:8} for the second inequality) 
\begin{equation}
\mu\leq\lim_{t\to\infty}\left(\inf_{\bx\in\CCX}
F_{\lambda_t}(\bx)\right)\leq\inf_{x\in\CCN}f(x).
\end{equation}
Taking $\mu\uparrow\inf_{x\in\CCN}f(x)$ completes the result.
\end{proof}

\section{Convergence of Algorithm~\ref{alg:scg}}
\label{sec:3}
We first prove that Algorithm~\ref{alg:scg} converges in function
value when $f$ is convex (Section~\ref{sec:cvx}). Then, we
establish guarantees for stationarity in general
(Section~\ref{sec:ncv}). We begin with an estimate which is used
for both settings.

\begin{lemma}
\label{l:gbound}
Let $(C_i)_{i\in I}$ be a finite collection of nonempty compact
convex subsets of $\HH$ with diameters $\{R_i\}_{i\in
I}\subset\left[0,+\infty\right[$, and let $\DD$ denote
the diagonal subspace of $\HHH$. Suppose that $\CCN\neq\varnothing$. Then
\begin{equation}
\left(\forall
\bx,\by\in\CCX\right)\quad\dist_{\DD}^2(\bx)\leq\sum_{i\in
I}\omega_iR_i^2\quad\text{and}\quad \|\bx-\by\|^2\leq\sum_{i\in
I}\omega_iR_i^2.
\end{equation}
\end{lemma}
\begin{proof}
Since, for every $i\in I$, $\proj_{\CCN}(A\bx)\in C_i$,
\eqref{e:gbound} yields the upper bound 
\begin{equation}
\dist_{\DD}^2(\bx)\leq\sum_{i\in
I}\omega_i\|\bx^{i}-\proj_{\CCN}(A\bx)\|^2\leq\sum_{i\in
I}\omega_iR_i^2.
\end{equation}
For the second bound, $\|\bx-\by\|^2=\sum_{i\in
I}\omega_i\|\bx^i-\by^i\|^2\leq\sum_{i\in I}\omega_iR_i^2$.
\end{proof}


\subsection{Convex setting}
\label{sec:cvx}
Here we show that, if $f$ is convex, Algorithm~\ref{alg:scg}
achieves an $\mathcal{O}(\ln t/\sqrt{t})$ convergence rate in terms
of the primal value gap of our subproblems \eqref{e:49}. In tandem
with Proposition~\ref{p:conv}, this establishes function value
convergence. Unlike the Augmented Lagrangian approaches
\cite{Gide18,Fall19,Yurt19}, our analysis does not require
further assumptions concerning the relative interiors of
$(C_i)_{i\in I}$, making it consistent with traditional 
Frank-Wolfe theory \cite[Section~2.1]{CGsurvey}.

\begin{lemma}
\label{l:rec}
Let $f$ be convex and $L_f$-smooth, let $\DD$ denote the
diagonal subspace of $\HHH$, let $(C_i)_{i\in I}$ be a finite
collection of nonempty compact convex subsets of $\HH$ with
diameters $\{R_i\}_{i\in I}\subset\left[0,+\infty\right[$ 
such that $\CCN\neq\varnothing$, and for every $\lambda\geq 0$, set
$F_{\lambda}\colon\HHH\to\left]-\infty,+\infty\right]\colon\bx\mapsto
f(A\bx)+\lambda\dist_{\DD}^2(\bx)/2$, set
$\bx_t^*\in\Argmin_{\bx\in\CCX}F_{\lambda_t}(\bx)$, and set 
$H_{t}=F_{\lambda_t}(\bx_t)-F_{\lambda_t}(\bx_t^*)$.
Suppose that $(\lambda_t)_{t\in\NN}$ is an increasing sequence.
Then the iterates of Algorithm~\ref{alg:scg} satisfy 
\begin{equation}
\label{e:rec}
H_{t+1}\leq
(1-\gamma_t)
H_t
+\frac{(\lambda_{t+1}-\lambda_t)}{2}\sum_{i\in I}\omega_iR_i^2
+\gamma_t^2\frac{(\lambda_t+L_f)}{2}\sum_{i\in I}\omega_iR_i^2.
\end{equation}
\end{lemma}
\begin{proof}
Let us begin by 
observing that $F_{\lambda_t}$ is convex and
$L_f+\lambda_t$-smooth (cf. Fact~\ref{f:gsmooth}).
Since Algorithm~\ref{alg:scg} performs one step of
the vanilla CG algorithm to \eqref{e:49}, a standard CG argument
\cite{CGsurvey} (relying on smoothness \eqref{e:2},
Line~\ref{a:lmo} and Fact~\ref{f:22}, then convexity \eqref{e:1})
shows
\begin{align}
F_{\lambda_{t}}(\bx_{t+1})-F_{\lambda_{t}}(\bx_{t})
\label{e:4213}
&\leq\gamma_t\Big(F_{\lambda_t}(\bx_t^*)-F_{\lambda_t}(\bx_t)\Big)
+\gamma_t^2\frac{L_f+\lambda_t}{2}\sum_{i\in I}\omega_iR_i^2.
\end{align}
Using Lemma~\ref{l:Frel}, then adding
$F_{\lambda_t}(\bx_t)-F_{\lambda_t}(\bx_t^*)$ to both sides of
\eqref{e:4213} reveals
\begin{align}
H_{t+1}&\leq
F_{\lambda_{t+1}}(\bx_{t+1})-F_{\lambda_{t}}(\bx_t^*)\\
&=F_{\lambda_{t}}(\bx_{t+1})-F_{\lambda_{t}}(\bx_t^*)
+\frac{\lambda_{t+1}-\lambda_t}{2}\dist_{\DD}^2(\bx_{t+1})\\
&\leq(1-\gamma_t)
H_{t}
+\frac{\lambda_{t+1}-\lambda_t}{2}\dist_{\DD}^2(\bx_{t+1})
+\gamma_t^2\frac{L_f+\lambda_t}{2}\sum_{i\in I}\omega_iR_i^2.
\end{align}
Finally, Lemma~\ref{l:gbound} finishes the result.
\end{proof}

\begin{theorem}
\label{t:convex}
In the setting of Lemma~\ref{l:rec}, for every $t\geq 0$ set
$\gamma_t=2/(\sqrt{t}+2)$. Let $\lambda_0>0$ and for every $t\geq
1$ set $\lambda_{t+1}=\lambda_t+\lambda_0(\sqrt{t}+2)^{-2}$.
Then, for every $t\in\NN$, the iterates of Algorithm~\ref{alg:scg}
satisfy
\begin{equation}
\label{e:rate}
0\leq
H_t\leq
2\sum_{i\in I}\omega_iR_i^2\left(
\frac{\lambda_0(2\ln(\sqrt{t}+2)+\frac{1}{4})+L_f}{\sqrt{t}+2}
+\frac{4\lambda_0}{(\sqrt{t}+2)^2}
\right).
\end{equation}
In particular, $F_{\lambda_t}(\bx_t)\to\inf_{x\in\CCN}f(x)$ and
$\dist_{\DD}(\bx_t)\to 0$. Furthermore, every accumulation point
$\bx_{\infty}$ of $(\bx_t)_{t\in\NN}$ produces a solution
$A\bx_{\infty}\in\CCN$ such that
$f(A\bx_{\infty})=\inf_{x\in\CCN}f(x)$.
\end{theorem}
\begin{proof}
For notational convenience, set $R=\sum_{i\in I}\omega_iR_i^2$
and
$\xi\colon\RR\to\RR\colon
s\mapsto2\ln(\sqrt{s}+2)+4/(\sqrt{s}+2)$.
By calculus, for every $t\in\NN$ such that $t\geq 1$,
$\lambda_t-\lambda_0\leq\lambda_0\xi(t)-\lambda_0\xi(0)$, so
$\lambda_t\leq\lambda_0\xi(t)$.
We shall proceed by induction. The base case for $t=0$ follows
from \eqref{e:2}, \eqref{e:opt}, and Lemma~\ref{l:gbound}. 
Next, we suppose that \eqref{e:rate} holds for $t\in\NN$.
Our inductive hypothesis, bound on $\lambda_t$, and \eqref{e:rec} 
yield
\begin{align}
H_{t+1}&\leq
(1-\gamma_t)\left(
2R\frac{\lambda_0\xi(t)+L_f+\frac{\lambda_0}{4}}{\sqrt{t}+2}
\right)
+\frac{\lambda_{t+1}-\lambda_t}{2}R
+\gamma_t^2\frac{(L_f+\lambda_0\xi(t))R}{2}\\
&=\frac{\sqrt{t}}{\sqrt{t}+2}\left(
2R\frac{\lambda_0\xi(t)+L_f+\frac{\lambda_0}{4}}{\sqrt{t}+2}
\right)
+2R\left(\frac{L_f+\frac{\lambda_0}{4}}{(\sqrt{t}+2)^2}
+\frac{\lambda_0\xi(t)
}{(\sqrt{t}+2)^2}\right)\\
\label{e:penultimate}
&\leq\frac{\sqrt{t}+1}{(\sqrt{t}+2)^2}
\big(
2R(L_f+\frac{\lambda_0}{4}+\lambda_0\xi(t+1))
\big)\\
\label{e:f12}
&\leq \frac{1}{\sqrt{t+1}+2}
\left(
2R(L_f+\frac{\lambda_0}{4}+\lambda_0\xi(t+1))
\right),
\end{align}
where \eqref{e:penultimate} is because $\xi$ is increasing 
and \eqref{e:f12} is because and
$(\sqrt{t}+1)(\sqrt{t+1}+2)\leq(\sqrt{t}+2)^2$. Having shown
\eqref{e:rate}, we point out that Proposition~\ref{p:fconv} implies
$\lim_{t\to\infty}F_{\lambda_t}(\bx_t^*)=\inf_{x\in\CCN}f(x)$. 
Hence $\lim_{t\to\infty}F_{\lambda_t}(\bx_t)$ exists and, via
\eqref{e:rate}, is equal to $\inf_{x\in\CCN}f(x)$. Since
$\lambda_t\to\infty$, it must be that $\dist_{\DD}^2(\bx_t)\to 0$.
Therefore, every accumulation point $x_{\infty}\in\CCX$ must also
reside in $\DD$, so $A\bx_{\infty}\in\CCN$. Passing to a
subsequence, since $f$ is continuous we have
\begin{equation}
\inf_{x\in\CCN}f(x)\leq f(A\bx_{\infty})=
\lim_{k\to\infty}f(A\bx_{t_k})\leq
\lim_{k\to\infty}F_{\lambda_{t_k}}(\bx_{t_k})
=\inf_{x\in\CCN}f(x).
\end{equation}
%
\end{proof}

Note that, although Theorem~\ref{t:convex} shows convergence of the
primal gaps of the subproblem \eqref{e:49}, these gaps are never
actually computed in practice, since $\bx_{t}^*$ is inaccessible.
We also point out that, for the choice of $\lambda_0=L_f$, our
convergence rate becomes scale-invariant.

The convergence rate in Theorem~\ref{t:convex} is atypical of CG
algorithms with convex objective functions, because they usually
have an $\mathcal{O}(1/t)$ convergence rate. 
This was achieved in the split-LMO setting under the condition
$m=2$ in \cite{Brau22,MuZh16} and with a Slater-type condition in
\cite{Gide18} by choosing stepsizes of magnitude
$\gamma_t=\mathcal{O}(1/t)$. However, in order to achieve
convergence in the proof of Theorem~\ref{t:convex} with this larger
stepsize, this would necessitate that
$\lambda_{t+1}-\lambda_t\leq\mathcal{O}(1/t^2)$, i.e.,
$\lambda_t\not\to\infty$. Since Example~\ref{ex:linf} establishes
that $\lambda_t\to\infty$ can be necessary (supported also by
Proposition~\ref{p:conv}), we would no longer be able to show that
the sequence of relaxed subproblems \eqref{e:49} converges to the
original splitting problem \eqref{e:p}. So, using a faster stepsize
schedule would still yield a convergent algorithm, but it
would not necessarily solve \eqref{e:p}. We shall consider the
topic of achieving a faster rate with extra assumptions in future
work.

\begin{remark}
\label{r:iterates}
Without additional assumptions, Algorithm~\ref{alg:scg} does not
guarantee iterate convergence of $(x_t)_{t\in\NN}$, which is
consistent with other CG methods \cite{Bolt22}. If, for instance,
$f$ is also
$\mu$-strongly convex, then Theorem~\ref{t:convex} can be
strengthened to provide convergence of the averages, because
$A\bx_t^*$ converges to the unique
solution $x^*$ of \eqref{e:p} and
$0\leq\mu\|A\bx_t-A\bx_t^*\|/2\leq
F_{\lambda_t}(\bx_t)-F_{\lambda_t}(\bx_t^*)\to 0$, so
$A\bx_t\to x^*$ as well.
\end{remark}

\subsection{Nonconvex setting}
\label{sec:ncv}

For CG methods which address \eqref{e:OG} in the case when $f$ is
nonconvex, it is standard to show that the Frank-Wolfe
gap at $x\in\HH$, $G_{f,C}(x):=\sup_{v\in C}\scal{\nabla f(x)}{x-v}$,
converges to zero, because $f$ is stationary at $x\in C$ whenever
the \FW gap vanishes \eqref{e:opt} \cite{CGsurvey}.
Since \FW gaps are highly variable between iterations, convergence
rates are typically derived for the average of \FW gaps. In this
section, we consider the \FW gaps for our subproblems \eqref{e:49}
which converge to \eqref{e:p} (in the sense of
Proposition~\ref{p:conv}).

We begin by connecting the \FW gaps of our
subproblems \eqref{e:49} to that of the original problem \eqref{e:p}.
In particular, for every $\lambda\geq 0$ the
Frank-Wolfe gaps of our subproblems at $\bx\in\CCX$ provide an
upper bound to {\em both} the penalty $\lambda\dist_{\DD}^2(\bx)$
and the \FW gap of the original problem \eqref{e:p} at $A\bx$.
Although
$G_{F_{\lambda},\CCX}(\bx_t)\geq 0$ is guaranteed, it is
interesting to note that the \FW gap for
the splitting problem \eqref{e:p}, namely $G_{f,\CCN}(A\bx_t)$, may
actually be negative since $A\bx_t$ is not guaranteed to reside in
$\CCN$ after a finite number of iterations.

\begin{lemma}
\label{l:gapbound}
Let $f$ be smooth, set $\beta_f=\sup_{\bx\in\CCX}\|\nabla f(A\bx)\|$,
let $\DD\subset\HHH$ denote the diagonal subspace of $\HHH$, let
$(C_i)_{i\in I}$ be a finite collection of nonempty compact convex
subsets of $\HH$ with diameters $\{R_i\}_{i\in
I}\subset\left[0,+\infty\right[$ such that
$\CCN\neq\varnothing$, and for every $\lambda\geq 0$, set
$F_{\lambda}\colon\HHH\to\left]-\infty,+\infty\right]\colon\bx\mapsto
f(A\bx)+\lambda\dist_{\DD}^2(\bx)/2$. Then, for every $\bx\in\CCX$,
\begin{equation}
\begin{aligned}
\label{e:gaps}
\sup_{\bv\in\bigtimes_{i\in I}C_i}\scal{\nabla
F_{\lambda}(\bx)}{\bx-\bv}& \geq\sup_{v\in\bigcap_{i\in I}C_i}
\scal{\nabla f(A\bx)}{A\bx-v}+\lambda \dist_{\DD}^2(\bx)\\
&\geq
-\beta_f\sum_{i\in I}\omega_iR_i.
\end{aligned}
\end{equation}
\end{lemma}
\begin{proof}
First, by infimizing over a subset of $\CCX$, we find
\begin{align}
\inf_{\bv\in\bigtimes_{i\in I}C_i}\scal{\nabla
F_{\lambda}(\bx)}{\bv-\bx}
&=\inf_{\bv\in\bigtimes_{i\in I}C_i}
\scal{A^*\nabla f(A\bx)+\lambda(\bx-A^*A\bx)}{\bv-\bx}\\
&
\begin{multlined}
\leq\inf_{\bv\in\DD\cap\CCX}
\scal{\nabla
f(A\bx)}{A\bv-A\bx}+\\
\lambda\scal{\bx-A^*A\bx}{\bv-\bx}.
\end{multlined}
\end{align}
Since $\bx-A^*A\bx\in\DD^{\perp}$ and $A^*A\bx\in\DD$, we have the
following identity for every
$\bv\in\DD$
\begin{equation}
\scal{\bx-A^*A\bx}{\bv-\bx}=\scal{\bx-A^*A\bx}{-A^*A\bx-(\bx-A^*A\bx)}
=-\|\bx-A^*A\bx\|^2.
\end{equation}
So, using Proposition~\ref{prop:decomp} for a change of variables, 
we set $p=\proj_{\CCN}(A\bx)$ to find that
\begin{align}
\label{e:93}
\inf_{\bv\in\bigtimes_{i\in I}C_i}\scal{\nabla
F_{\lambda}(\bx)}{\bv-\bx}
&\leq\inf_{v\in\bigcap_{i\in I}C_i}
\scal{\nabla f(A\bx)}{v-A\bx}-\lambda\dist_{\DD}^2(\bx)\\
&\leq \scal{\nabla f(A\bx)}{p
-A\bx
}
-\lambda\dist_{\DD}^2(\bx)\\
&\leq \beta_f\dist_{\bigcap_{i\in I}C_i}(A\bx)-\lambda
\dist_{\DD}^2(\bx)\\
&\leq\beta_f\sum_{i\in I}\omega_iR_i,
\end{align}
since $\dist_{\bigcap_{i\in I}C_i}(A\bx)
=\|\sum_{i\in I}\omega_i(\bx^i-p)\|\leq\sum_{i\in
I}\omega_iR_i$.
Finally, negation yields \eqref{e:gaps}.
\end{proof}

With these results in-hand, we can now prove our main result.

\begin{theorem}
\label{t:nconv}
Let $f$ be $L_f$-smooth, let $\DD\subset\HHH$ denote the
diagonal subspace of $\HHH$, let $(C_i)_{i\in I}$ be a finite
collection of nonempty compact convex subsets of $\HH$ with
diameters $\{R_i\}_{i\in I}\subset\left[0,+\infty\right[$ such that
$\CCN\neq\varnothing$, and for every $\lambda\geq 0$, set
$F_{\lambda}\colon\HHH\to\left]-\infty,+\infty\right]\colon\bx\mapsto
f(A\bx)+\lambda\dist_{\DD}^2(\bx)/2$. 
Set $\gamma_t=1/\sqrt{t+1}$, let $\lambda_0>0$, and for every $t\geq
1$, set $\lambda_t=\lambda_0\sum_{k=0}^{t-1}1/(k+1)$. Then, for every
$t\geq 1$, the iterates of Algorithm~\ref{alg:scg}
satisfy\footnote{Precise constants for \eqref{e:ncvrate} are in
\eqref{e:ncrate}.}
\begin{equation}
\label{e:ncvrate}
0\leq\frac{1}{t}\sum_{k=0}^{t-1}
\sup_{\bv\in\CCX}
\big\langle{\nabla
F_{\lambda_k}(\bx_k)}\;\big|\:{\bx_k-\bv}\big\rangle
\leq
\mathcal{O}\left(\frac{\ln t }{\sqrt{t}}+\frac{1}{\sqrt{t}}\right).
\end{equation}
In particular, there exists a subsequence $(t_k)_{k\in\NN}$ such
that 
\begin{equation}
\left(\sup_{\bv\in\CCX}\scal{\nabla
F_{\lambda_{t_k}}(\bx_{t_k})}{\bx_{t_k}-\bv}\right)_{k\in\NN}\to
0.\end{equation}
Furthermore, every accumulation point $\bx_{\infty}$ of
$(\bx_{t_k})_{k\in\NN}$ yields a point in the intersection via
$A\bx_{\infty}\in\CCN$ which satisfies $G_{f,\CCN}(A\bx_t)=0$,
i.e., $A\bx_{\infty}$ is a stationary point of the problem \eqref{e:p}.
\end{theorem}
\begin{proof}
For notational convenience, set $R=\sum_{i\in
I}\omega_iR_i^2$, $R_{A}=\sum_{i\in I}\omega_iR_i$, and 
$B=\max\{\beta_p\sqrt{R},R\}$; for every $t\in\NN$, 
let $\bx_t^*$ be a minimizer of $F_{\lambda_t}$ over $\CCX$, 
set $H_{t}=F_{\lambda_t}(\bx_t)-F_{\lambda_t}(\bx_t^*)$, and set
$\bv_t=(\bv_{t}^{i})_{i\in I}\in\CCX$ (cf. Line~\ref{a:lmo}).
Let us recall
that $F_{\lambda_t}$ is $(L_f+\lambda_t)$-smooth.
By the optimality of $\bv_t$ (Fact~\ref{f:22} and
Line~\ref{a:lmo}), construction in Line~\ref{a:new},
and the smoothness inequality \eqref{e:2}, we have
\begin{equation}
0\leq \gamma_t\scal{\nabla F_{\lambda_t}(\bx_t)}{\bx_t-\bv_t}\leq
F_{\lambda_t}(\bx_t)-F_{\lambda_t}(\bx_{t+1})+\gamma_t^2\frac{L_f+\lambda_t}{2}\|\bv_t-\bx_t\|^2.
\end{equation}
So, using Lemma~\ref{l:Frel} and Lemma~\ref{l:gbound} twice,
\begin{align}
0&\leq\scal{\nabla F_{\lambda_t}(\bx_t)}{\bx_t-\bv_t}\\
&\leq
\frac{F_{\lambda_t}(\bx_t)-F_{\lambda_{t+1}}(\bx_{t+1})}{\gamma_t}
+\frac{\lambda_{t+1} -\lambda_t}{\gamma_t}\dist_{\DD}^2(\bx_{t+1})
+\gamma_t\frac{L_f+\lambda_t}{2}R
\\
&\leq
\frac{F_{\lambda_t}(\bx_t)-F_{\lambda_{t+1}}(\bx_{t+1})}{\gamma_t}
+\frac{\lambda_{t+1} -\lambda_t}{\gamma_t}R
+\gamma_t\frac{L_f+\lambda_t}{2}R.
\end{align}
Furthermore, since $f$ and
$\dist_{\DD}^2/2$ are
smooth and $\sum_{i\in I}\omega_i C_i$ and $\CCX$ are
compact, it follows that their gradients are bounded. Hence, $f$
and $\dist_{\DD}^2/2$ are Lipschitz continuous on these sets, with 
constants $\beta_f:=\sup_{c\in\CCM}\|\nabla f(c)\|$ and
$\beta_p:=\sup_{\bc\in\bigtimes_{i\in I}C_i}\|\nabla
\dist_{\DD}^2(\bc)/2\|$ respectively. 
Therefore, 
we find that by Jensen's inequality and Lemma~\ref{l:gbound},
\begin{align}
H_{t}&\leq\beta_f\|A\bx_t-A\bx_t^*\|+
\lambda_t\beta_p\|\bx_t-\bx_t^*\|\\
&\leq
\beta_f\sum_{i\in I}\omega_iR_i+\lambda_t\beta_p\sqrt{\sum_{i\in
I}\omega_iR_i^2}\\
\label{e:lip2}
&=\beta_fR_A+\lambda_t\beta_p\sqrt{R}.
\end{align}
By Lemma~\ref{l:Frel}, we have
$\gamma_t^{-1}(F_{\lambda_{t+1}}(\bx_{t+1}^*)-F_{\lambda_{t}}(\bx_t^*))\geq
0$. Combining all of these facts, we find
\begin{align}
0 &\leq\sum_{k=0}^{t-1}\Scal{\nabla F_{\lambda_k}(\bx_k)}{\bx_k-\bv_k}\\
&\leq\sum_{k=0}^{t-1}
\left(
\frac{F_{\lambda_k}(\bx_k)-F_{\lambda_{k+1}}(\bx_{k+1})}{\gamma_k}
+\frac{\lambda_{k+1}-\lambda_k}{\gamma_k}R
+\gamma_k\frac{L_f+\lambda_k}{2}R\right)\\
\label{e:807}
&\leq\sum_{k=0}^{t-1}
\left(
\frac{H_{k}-H_{k+1}}{\gamma_k}
+\frac{\lambda_{k+1}-\lambda_k}{\gamma_k}R
+\gamma_k\frac{L_f+\lambda_k}{2}R\right)
\end{align}
where we use Lemma~\ref{l:Frel} in \eqref{e:807}. However, the
upper bound in \eqref{e:807} is equal to
\begin{equation}
\label{e:8080}
\frac{H_{0}}{\gamma_0}-\frac{H_{t}}{\gamma_{t-1}}
+\sum_{k=1}^{t-1}\left(\frac{1}{\gamma_k}-\frac{1}{\gamma_{k-1}}
\right)H_{k}
+\sum_{k=0}^{t-1}\left(\frac{\lambda_{k+1} -\lambda_k}{\gamma_k}R
+\gamma_k\frac{L_f+\lambda_k}{2}R
\right).
\end{equation}
Continuing from \eqref{e:807}, we drop a negative term and
use \eqref{e:lip2} in \eqref{e:8080}; we then use the construction
$B$, and simplify to reveal
\begin{align}
\label{e:808}
0 &\leq\sum_{k=0}^{t-1}\Scal{\nabla F_{\lambda_k}(\bx_k)}{\bx_k-\bv_k}\\
&\begin{multlined}\leq
\frac{\beta_fR_A+\lambda_0\beta_p\sqrt{R}}{\gamma_0}
+
\sum_{k=1}^{t-1}
\left(\frac{1}{\gamma_k}-\frac{1}{\gamma_{k-1}}\right)
\left(\beta_fR_A+\lambda_k\beta_p\sqrt{R}\right)
\\
+\sum_{k=0}^{t-1}\left(\frac{\lambda_{k+1}
-\lambda_k}{\gamma_k}R
+\gamma_k\frac{L_f+\lambda_k}{2}R\right)
\end{multlined}\\
\label{e:901}
&\begin{multlined}\leq\frac{
\beta_fR_A+\lambda_0B}{\gamma_0}
+
\sum_{k=1}^{t-1}
\left(\frac{1}{\gamma_k}-\frac{1}{\gamma_{k-1}}\right)
\left(\beta_fR_A+\lambda_kB\right)
+\sum_{k=0}^{t-1}
\frac{\lambda_{k+1}-\lambda_k}{\gamma_k}B
\\
+\sum_{k=0}^{t-1}
\gamma_k\frac{L_f+\lambda_k}{2}R
\end{multlined}\\
\label{e:902}
&=
\frac{\beta_fR_A+\lambda_{t}B}{\gamma_{t-1}}
+\sum_{k=0}^{t-1}
\gamma_k\frac{L_f+\lambda_k}{2}R.
\end{align}
Next, we note that $\sum_{k=0}^{t-1}\gamma_k\leq2\sqrt{t}$ and
$\lambda_{t}\leq\lambda_0(\ln(t+1)+1)$, so 
\begin{align}
0&\leq\frac{1}{t}\sum_{k=0}^{t-1}\Scal{\nabla F_{\lambda_k}(\bx_k)}{\bx_k-\bv_k}\\
&\leq\frac{\beta_fR_A+\lambda_{t}B}{\sqrt{t}}
+\frac{1}{t}\sum_{k=0}^{t-1}\gamma_k\frac{L_f+\lambda_k}{2}R\\
&\leq\frac{\beta_fR_A+\lambda_{t}B}{\sqrt{t}}
+\frac{1}{t}(L_f+\lambda_{t-1})\sum_{k=0}^{t-1}\gamma_k\frac{1}{2}R\\
&\leq\frac{\beta_fR_A+\lambda_0(\ln(t+1)+1)B}{\sqrt{t}}
+\frac{1}{\sqrt{t}}(L_f+\lambda_0(\ln(t)+1))R\\
&\begin{multlined}
\leq\frac{1}{\sqrt{t}}\Big(\beta_f\sum_{i\in I}\omega_iR_i+
(L_f+\lambda_0)\sum_{i\in I}\omega_iR_i^2
+\lambda_0B\Big)\\
+\frac{\ln(t+1)}{\sqrt{t}}\lambda_0\Big(\sum_{i\in
I}\omega_iR_i^2+B\Big),
\end{multlined}
\label{e:ncrate}
\end{align}
which establishes \eqref{e:ncvrate}. Since the Frank-Wolfe gaps
$(\scal{\nabla F_{\lambda_t}(\bx_t)}{\bx_t-\bv_t})_{t\in\NN}$ are
positive and the
sequence of averages goes to zero, the existence of a subsequence
$(t_k)_{k\in\NN}$ such that $\scal{\nabla
F_{\lambda_t}(\bx_{t_k})}{\bx_{t_k}-\bv_{t_k}}\to 0$ follows. 
Lemma~\ref{l:gapbound} implies that 
\begin{equation}
\left(\sup_{v\in\CCN}\scal{\nabla
f(A\bx_{t_k})}{A\bx_{t_k}-v}+
\lambda_{t_k}\dist_{\DD}^2(\bx_{t_k})\right)_{k\in\NN}
\end{equation}
is bounded. So, since $\lambda_{t_k}\to\infty$, we must have
$\dist_{\DD}^2(\bx_{t_k})\to 0$. Therefore, for every accumulation
point
$\bx_{\infty}$ of $(\bx_{t_k})_{k\in\NN}$,
$\bx_{\infty}\in\DD\cap\CCX$, so $A\bx_{\infty}\in\CCN$ and
\begin{equation}
0\leq\sup_{v\in\CCN}
\scal{\nabla f(A\bx_{\infty})}{A\bx_{\infty}-\bv}.
\end{equation}
Finally, we can
bound the gap above using continuity and Lemma~\ref{l:gapbound}:
\begin{align}
\sup_{v\in\CCN}\scal{\nabla f(A\bx_{\infty})}{A\bx_{\infty}-\bv}
&\leq
\limsup_{k\to\infty}\left(
\sup_{v\in\CCN}\scal{\nabla f(A\bx_{t_k})}{A\bx_{t_k}-\bv}\right)\\
&\leq
\limsup_{k\to\infty}\left(
\scal{\nabla
F_{\lambda_{t_k}}(\bx_{t_k})}{\bx_{t_k}-\bv_{t_k}}\right)\\
&= 0.
\end{align}
Since $G_{f,\CCN}(A\bx_{\infty})=0$, we conclude from \eqref{e:opt}
that $A\bx_{\infty}$ is a stationary point.
\end{proof}

\begin{remark}
\label{r:verify}
We emphasize that, for the cost of one extra inner product, 
the Frank-Wolfe gap $\scal{\nabla
F_{\lambda_t}(\bx_t)}{\bx_t-\bv_t}$ can be computed while
Algorithm~\ref{alg:scg} is running. So, checking for
stationarity in the subproblems \eqref{e:49} is tractable in
practice. Also, similarly to the convex-case, the choice of
$\lambda_0=L_f$ makes our convergence rate in \eqref{e:ncrate}
scale-invariant.
\end{remark}

\section{Conclusion and Future Work}

Theorem~\ref{t:nconv} appears to be the first convergence guarantee
for solving \eqref{e:p} in the nonconvex split-LMO setting.
Furthermore, our rate of convergence is only one log factor less
than the rate of CG for one set constraint ($m=1$) \cite{Pedr20}.
While it is unclear if this log factor can be removed for the
nonconvex setting, we believe that the analysis for the convex rate
can be improved since typically the nonconvex average-F-W-gap
rate is quadratically slower than the convex primal gap rate
\cite{Pedr20}. This speed-up has been achieved in some settings
with algorithms which require one LMO call per iteration
\cite{Gide18,MuZh16}, but it appears that the question of whether
or not $\mathcal{O}(1/t)$ convergence is possible in the split-LMO
setting without additional assumptions remains open.

While our analysis shows that Algorithm~\ref{alg:scg} 
asymptotically solves \eqref{e:p} (in the sense of
Theorems~\ref{t:convex} and \ref{t:nconv}), one drawback is that
the {\em rates} of convergence concern the penalized functions
$F_{\lambda_t}$, which contain both the primal function value and
the penalized feasibility term $\lambda_t\dist_{\DD}^2$. This means
that, outside of the case $m=1$, both
primal suboptimality and feasibility are analyzed in one, composite
quantity. An improvement to our analysis would be to translate
convergence rates of the composite penalty into separate
convergence rates for both primal suboptimality and feasibility,
e.g., as was done for a method of sequential averaging
in the recent preprint \cite{Tran23}.

Based on preliminary numerical experiments, there may be
more to discover for Algorithm~\ref{alg:scg}. We have observed that the algorithm
performance (in terms of feasibility and \FW gap) can be highly
dependent on the initial value $\lambda_0$ when using the parameter
schedules
$(\lambda_t,\gamma_t)=(\mathcal{O}(\ln t), \mathcal{O}(1/\sqrt{t}))$
which are proven to work. However, based on rough experimentation,
we are hopeful that it may be possible to use
Algorithm~\ref{alg:scg} with an adaptive strategy for
$(\lambda_t)_{t\in\NN}$ in conjunction with a faster
stepsize $(\gamma_t)_{t\in\NN}$, for instance, 
$\mathcal{O}(1/t)$ or short-step selections similar to
\cite{Pedr20}. In fact, the proofs of Theorems~\ref{t:convex} and
\ref{t:nconv} can easily be extended to
a short-step selection for $\gamma_t$ by minimizing the upper bound
arising from \eqref{e:2}; however, thus-far, we have not been able
to properly analyze a new adaptive scheme for
$(\lambda_t)_{t\in\NN}$. This is a topic of future work. At least
for now, it appears our contribution is predominantly of
theoretical interest.

In addition to the questions above, there are several interesting
theoretical and numerical investigations to be performed.
For instance, one could investigate
Algorithm~\ref{alg:scg} under additional assumptions on the
objective or constraints. For instance, CG algorithms
possess accelerated convergence rates when
the objective function or constraints are strongly convex
\cite{Garb15,Wirt23}; one could extend this analysis to
Algorithm~\ref{alg:scg}.
Arguments similar to that of \cite[Theorem~2.6]{CGsurvey} appear
fruitful for allowing a variant of Algorithm~\ref{alg:scg} which
admits approximate LMO evaluations in the convex setting of
Theorem~\ref{t:convex}, provided the accuracy of computing $\bv_t$
is bounded by $\mathcal{O}(1/\sqrt{t})$; more generally
investigating approximate LMO implementations for this setting
appears to be an open problem. Many projection-based
splitting methods have an advantage of being
block-iterative, i.e., instead of requiring a computation for all
constraints indexed by $I$ (as is required in the for loop in
Algorithm~\ref{alg:scg}, Line~\ref{a:block}) at every iteration
$t$, only a subset $I_t\subset I$ of updates are performed. This
can significantly reduce the computational load per iteration, and
block-iterative projection methods enjoy convergence under very
mild assumptions on the blocks $(I_t)_{t\in\NN}$
\cite{Comb11,Comb21}. It is worth noting that the inner loop of
Algorithm~\ref{alg:scg} can be
parallelized, and a block-iterative capability would further
improve the per-iteration cost. Several 
LMO-based block-iterative algorithms have been proposed for solving
problems like the relaxation \eqref{e:49} \cite{Beck15,Bomz24}, but
extending them to solve \eqref{e:p} remains to be done.

\section*{Acknowledgments}
The work for this article has been supported by MODAL-Synlab, and
took place on the Research Campus MODAL funded by the German
Federal Ministry of Education and Research (BMBF) (fund numbers
05M14ZAM, 05M20ZBM).
This research was also supported by the DFG Cluster of
Excellence MATH+ (EXC-2046/1, project ID 390685689) funded by the
Deutsche Forschungsgemeinschaft (DFG).

We thank the reviewers who have significantly improved
the quality of this article through their thoughtful comments.
We also thank Kamiar Asgari, G\'abor Braun, Mathieu Besan\c con,
Ibrahim Ozaslan, Christophe Roux, Antonio Silveti-Falls, David
Mart\'inez-Rubio, and Elias Wirth for their valuable feedback and
discussions.

\bibliographystyle{siamplain}
\bibliography{scgarxiv4.bib}

\begin{thebibliography}{10}

\bibitem{Livre1}
{\sc H.~H. Bauschke and P.~L. Combettes}, {\em Convex Analysis and Monotone
  Operator Theory in Hilbert Spaces, 2nd ed.}, Springer, 2017.

\bibitem{Beck15}
{\sc A.~Beck, E.~Pauwels, and S.~Sabach}, {\em The cyclic block conditional
  gradient method for convex optimization problems}, SIAM J. Optim., 25 (2015),
  pp.~2024--2049.

\bibitem{Bern09}
{\sc T.~Bernholt, F.~Eisenbrand, and T.~Hofmeister}, {\em Constrained
  {M}inkowski sums: A geometric framework for solving interval problems in
  computational biology efficiently}, Discrete Comput. Geom., 42 (2009),
  pp.~22--36.

\bibitem{Bolt22}
{\sc J.~Bolte, C.~W. Combettes, and E.~Pauwels}, {\em The iterates of the
  {F}rank-{W}olfe algorithm may not converge}, Math. Oper. Res.,  (to appear).

\bibitem{Bomz24}
{\sc I.~Bomze, F.~Rinaldi, and D.~Zeffiro}, {\em Projection free methods on
  product domains}, Comput. Optim. Appl.,  (2024),
  \url{https://doi.org/10.1007/s10589-024-00585-5}.

\bibitem{CGsurvey}
{\sc G.~Braun, A.~Carderera, C.~Combettes, H.~Hassani, A.~Karbasi, A.~Mokhtari,
  and S.~Pokutta}, {\em Conditional gradient methods}, 2022,
  \url{https://arxiv.org/abs/2211.14103}.

\bibitem{Brau22}
{\sc G.~Braun, S.~Pokutta, and R.~Weismantel}, {\em Alternating linear
  minimization: Revisiting von {N}eumann’s alternating projections}, 2022,
  \url{https://arxiv.org/abs/2212.02933}.

\bibitem{Cens15}
{\sc Y.~Censor and A.~Cegielski}, {\em Projection methods: an annotated
  bibliography of books and reviews}, Optimization, 64 (2015), pp.~2343--2358,
  \url{https://doi.org/10.1080/02331934.2014.957701}.

\bibitem{Chry97}
{\sc I.~Chryssoverghi, A.~Bacopoulos, B.~Kokkinis, and J.~Coletsos}, {\em Mixed
  {F}rank--{W}olfe penalty method with applications to nonconvex optimal
  control problems}, J. Optim. Theory Appl., 94 (1997), pp.~311--334.

\bibitem{Comb21LMO}
{\sc C.~W. Combettes and S.~Pokutta}, {\em Complexity of linear minimization
  and projection on some sets}, Oper. Res. Lett., 49 (2021), pp.~565--571.

\bibitem{Comb11}
{\sc P.~L. Combettes and J.-C. Pesquet}, {\em Proximal Splitting Methods in
  Signal Processing}, Springer New York, New York, NY, 2011,
  \url{https://doi.org/10.1007/978-1-4419-9569-8_10}.

\bibitem{Comb21}
{\sc P.~L. Combettes and Z.~C. Woodstock}, {\em Reconstruction of functions
  from prescribed proximal points}, J. Approx. Theory, 268 (2021), p.~105606,
  \url{https://doi.org/https://doi.org/10.1016/j.jat.2021.105606}.

\bibitem{Ding22}
{\sc T.~Ding, D.~Lim, R.~Vidal, and B.~D. Haeffele}, {\em Understanding doubly
  stochastic clustering}, in Proceedings of the 39th International Conference
  on Machine Learning, K.~Chaudhuri, S.~Jegelka, L.~Song, C.~Szepesvari,
  G.~Niu, and S.~Sabato, eds., vol.~162 of Proceedings of Machine Learning
  Research, PMLR, 17--23 Jul 2022, pp.~5153--5165.

\bibitem{Duan19}
{\sc L.~L. Duan, A.~L. Young, A.~Nishimura, and D.~B. Dunson}, {\em {Bayesian
  constraint relaxation}}, Biometrika, 107 (2019), pp.~191--204,
  \url{https://doi.org/10.1093/biomet/asz069}.

\bibitem{Garb15}
{\sc D.~Garber and E.~Hazan}, {\em Faster rates for the {F}rank-{W}olfe method
  over strongly-convex sets}, in Proceedings of the 32nd International
  Conference on Machine Learning, F.~Bach and D.~Blei, eds., vol.~37 of
  Proceedings of Machine Learning Research, Lille, France, 07--09 Jul 2015,
  PMLR, pp.~541--549.

\bibitem{Garb21}
{\sc D.~Garber, A.~Kaplan, and S.~Sabach}, {\em Improved complexities of
  conditional gradient-type methods with applications to robust matrix recovery
  problems}, Math. Program., 186 (2021), pp.~185--208.

\bibitem{Tran23}
{\sc K.-H. Giang-Tran, N.~Ho-Nguyen, and D.~Lee}, {\em Projection-free methods
  for solving convex bilevel optimization problems}, 2023,
  \url{https://arxiv.org/abs/2311.09738}.

\bibitem{Gide18}
{\sc G.~Gidel, F.~Pedregosa, and S.~Lacoste-Julien}, {\em {F}rank-{W}olfe
  splitting via augmented {L}agrangian method}, in Proceedings of the 21st
  International Conference on Artificial Intelligence and Statistics,
  A.~Storkey and F.~Perez-Cruz, eds., vol.~84 of Proceedings of Machine
  Learning Research, PMLR, 09--11 Apr 2018, pp.~1456--1465.

\bibitem{HeHa15}
{\sc N.~He and Z.~Harchaoui}, {\em Semi-proximal mirror-prox for nonsmooth
  composite minimization}, in Advances in Neural Information Processing
  Systems, C.~Cortes, N.~Lawrence, D.~Lee, M.~Sugiyama, and R.~Garnett, eds.,
  vol.~28, Curran Associates, Inc., 2015.

\bibitem{Kolm21}
{\sc V.~Kolmogorov and T.~Pock}, {\em One-sided {F}rank-{W}olfe algorithms for
  saddle problems}, in Proceedings of the 38th International Conference on
  Machine Learning, M.~Meila and T.~Zhang, eds., vol.~139 of Proceedings of
  Machine Learning Research, PMLR, 18--24 Jul 2021, pp.~5665--5675.

\bibitem{LanR21}
{\sc G.~Lan, E.~Romeijn, and Z.~Zhou}, {\em Conditional gradient methods for
  convex optimization with general affine and nonlinear constraints}, SIAM J.
  Optim., 31 (2021), pp.~2307--2339.

\bibitem{WonX19}
{\sc K.~Lange, J.-H. Won, and J.~Xu}, {\em Projection onto {M}inkowski sums
  with application to constrained learning}, in Proceedings of the 36th
  International Conference on Machine Learning, K.~Chaudhuri and
  R.~Salakhutdinov, eds., vol.~97 of Proceedings of Machine Learning Research,
  PMLR, 09--15 Jun 2019, pp.~3642--3651.

\bibitem{LiuL19}
{\sc Y.-F. Liu, X.~Liu, and S.~Ma}, {\em On the nonergodic convergence rate of
  an inexact augmented {L}agrangian framework for composite convex
  programming}, Math. Oper. Res., 44 (2019), pp.~632--650.

\bibitem{Loza79}
{\sc T.~Lozano-P{\'e}rez and M.~A. Wesley}, {\em An algorithm for planning
  collision-free paths among polyhedral obstacles}, Commun. ACM, 22 (1979),
  pp.~560--570.

\bibitem{Migd94}
{\sc A.~Migdalas}, {\em A regularization of the {F}rank—{W}olfe method and
  unification of certain nonlinear programming methods}, Math. Program., 65
  (1994), pp.~331--345.

\bibitem{Mill21}
{\sc R.~D. Millán, O.~P. Ferreira, and L.~F. Prudente}, {\em Alternating
  conditional gradient method for convex feasibility problems}, 80 (2021),
  pp.~245--–269, \url{https://doi.org/10.1080/10556788.2013.796683}.

\bibitem{MuZh16}
{\sc C.~Mu, Y.~Zhang, J.~Wright, and D.~Goldfarb}, {\em Scalable robust matrix
  recovery: {F}rank-{W}olfe meets proximal methods}, SIAM J. Sci. Comput., 38
  (2016), pp.~A3291--A3317.

\bibitem{Pedr20}
{\sc F.~Pedregosa, G.~Negiar, A.~Askari, and M.~Jaggi}, {\em Linearly
  convergent {F}rank-{W}olfe with backtracking line-search}, in International
  conference on artificial intelligence and statistics, PMLR, 2020, pp.~1--10.

\bibitem{Pier84}
{\sc G.~Pierra}, {\em Decomposition through formalization in a product space},
  Math. Program., 28 (1984), pp.~96--115.

\bibitem{Rich12}
{\sc E.~Richard, P.~Savalle, and N.~Vayatis}, {\em Estimation of simultaneously
  sparse and low rank matrices}, in Proceedings of the 29th International
  Conference on Machine Learning, {ICML} 2012, Edinburgh, Scotland, UK, June 26
  - July 1, 2012, icml.cc / Omnipress, 2012.

\bibitem{Rock09}
{\sc R.~T. Rockafellar and R.~J.-B. Wets}, {\em Variational Analysis},
  vol.~317, Springer Science \& Business Media, 2009.

\bibitem{Roth17}
{\sc T.~Rothvoss}, {\em The matching polytope has exponential extension
  complexity}, J. ACM, 64 (2017), \url{https://doi.org/10.1145/3127497}.

\bibitem{Fall19}
{\sc A.~Silveti-Falls, C.~Molinari, and J.~Fadili}, {\em Generalized
  conditional gradient with augmented {L}agrangian for composite minimization},
  SIAM J. Optim., 30 (2020), pp.~2687--2725,
  \url{https://doi.org/10.1137/19M1240460}.

\bibitem{Wirt23}
{\sc E.~Wirth, T.~Kerdreux, and S.~Pokutta}, {\em Acceleration of
  {F}rank-{W}olfe algorithms with open-loop step-sizes}, in Proceedings of The
  26th International Conference on Artificial Intelligence and Statistics,
  F.~Ruiz, J.~Dy, and J.-W. van~de Meent, eds., vol.~206 of Proceedings of
  Machine Learning Research, PMLR, 25--27 Apr 2023, pp.~77--100.

\bibitem{Yang16}
{\sc Z.~Yang, J.~Corander, and E.~Oja}, {\em Low-rank doubly stochastic matrix
  decomposition for cluster analysis}, J. Mach. Learn. Res., 17 (2016),
  pp.~6454--6478.

\bibitem{Yurt19}
{\sc A.~Yurtsever, O.~Fercoq, and V.~Cevher}, {\em A conditional-gradient-based
  augmented {L}agrangian framework}, in Proceedings of the 36th International
  Conference on Machine Learning, K.~Chaudhuri and R.~Salakhutdinov, eds.,
  vol.~97 of Proceedings of Machine Learning Research, PMLR, 09--15 Jun 2019,
  pp.~7272--7281.

\bibitem{Yurt18}
{\sc A.~Yurtsever, O.~Fercoq, F.~Locatello, and V.~Cevher}, {\em A conditional
  gradient framework for composite convex minimization with applications to
  semidefinite programming}, in Proceedings of the 35th International
  Conference on Machine Learning, J.~Dy and A.~Krause, eds., vol.~80 of
  Proceedings of Machine Learning Research, PMLR, 10--15 Jul 2018,
  pp.~5727--5736.

\bibitem{Zhan22}
{\sc Z.~Zhang, Z.~Zhai, and L.~Li}, {\em Graph refinement via simultaneously
  low-rank and sparse approximation}, SIAM J. Sci. Comput., 44 (2022),
  pp.~A1525--A1553.

\end{thebibliography}

\end{document}